\documentclass[11pt]{amsart}

 \usepackage[utf8]{inputenc}
 \usepackage[T1]{fontenc}

 \usepackage[english]{babel}

 \usepackage{amsmath,amsfonts,amssymb,ulem}
 \usepackage{mathrsfs}
 \usepackage{dsfont} 
 \usepackage{mathtools} 
 \usepackage{esint} 
 \usepackage{xfrac} 
 \usepackage{defs} 
 \usepackage{fonts} 

 \usepackage{amsthm}
 \usepackage[shortlabels]{enumitem} 
 \usepackage[colorlinks=true,urlcolor=blue, citecolor=red,linkcolor=blue,linktocpage,pdfpagelabels, bookmarksnumbered,bookmarksopen]{hyperref} 
 \usepackage[paper=a4paper,hmargin=1in,vmargin=1in, marginparwidth=0.8in]{geometry} 

 \usepackage[colorinlistoftodos,prependcaption,linecolor=blue,backgroundcolor=blue!25,bordercolor=blue,textsize=tiny]{todonotes}
 \usepackage{cite} 

plus 6pt minus 12pt
\everymath{\displaystyle}
\numberwithin{equation}{section}

\newtheorem{theorem}{Theorem}[section]
\newtheorem{proposition}[theorem]{Proposition}
\theoremstyle{definition}

\theoremstyle{plain}
\newtheorem{lemma}[theorem]{Lemma}
\newtheorem{corollary}[theorem]{Corollary}
\theoremstyle{remark}
\newtheorem{remark}[theorem]{Remark}

\title[Linearization and localization of nonconvex functionals]{Linearization and localization of nonconvex functionals motivated by nonlinear peridynamic models}
 \author{Tadele Mengesha}
 \address[Tadele Mengesha]{Department of Mathematics,
 University of Tennessee,
 Knoxville, TN 37996, USA}
 \email{mengesha@utk.edu}
 \author{James M. Scott}
 \address[James M. Scott]{Department of Applied Physics and Applied Mathematics,
 Columbia University,
 New York, NY 10025, USA}
 \email{jms2555@columbia.edu}
\begin{document}

\maketitle

\begin{abstract}

We consider a class of nonconvex energy functionals that lies in the framework of the peridynamics model of continuum mechanics. 
The energy densities are functions of a nonlocal strain that describes deformation based on pairwise interaction of material points, and as such are nonconvex with respect to nonlocal deformation.
We apply variational analysis to investigate the consistency of the effective behavior of these nonlocal nonconvex functionals with established classical and peridynamic models in two different regimes.
In the regime of small displacement, we show the model can be effectively described by its linearization. To be precise, we rigorously derive what is commonly called the linearized bond-based peridynamic functional as a $\Gamma$-limit of nonlinear functionals.
In the regime of vanishing nonlocality, the effective behavior the nonlocal nonconvex functionals is characterized by an integral representation, which is obtained via $\Gamma$-convergence with respect to the strong $L^p$ topology.
We also prove various properties of the density of the localized quasiconvex functional such as frame-indifference and coercivity. We demonstrate that the density vanishes on matrices whose singular values are less than or equal to one. 
These results confirm that the localization, in the context of $\Gamma$-convergence, of peridynamic-type energy functionals exhibit behavior quite different from classical hyperelastic energy functionals.
\end{abstract}

\section{Introduction}\label{sec:Intro}
Consider a material occupying a bounded domain $\Omega\subset \mathbb{R}^{d}$.   
According to the bond-based peridynamics model, formulated by S. Silling in   \cite{Silling2001,Silling2007}, a microelastic material can be intuitively thought of as a complex mass-spring system where every pair of points $\bx$ and $\by$ are connected by a spring that captures a possibly nonlinear relationship between force and displacement. The force in the spring depends on the bond $\boldsymbol{\xi} = \by-\bx$.  When the material is subject to  a deformation ${\bf v}:\Omega \to \mathbb{R}^{d}$, the microstretch in the direction of $\frac{\by-\bx}{|\by-\bx|}$ at $\bx$ is given by $
 |\bv(\by)-\bv(\bx)| 
$.  Denoting the vector difference quotient $\cD\bv(\by, \bx)=  \frac{\bv(\by)-\bv(\bx)} {|\by-\bx|}$, the quantity $|\cD\bv(\by, \bx)| = \frac{|\bv(\by)-\bv(\bx)|}{|\by-\bx|} $ represents the ratio of the change in the length of the bond $\by-\bx$
 due to the deformation.  The pairwise {\it microstrain} or {\it microelongation} of the spring bond connecting $\bx$ and $\by$ is then given by 
 \[
 s[{\bv}](\by, \bx) :=  |\cD\bv(\by, \bx)|-1,
 \]
or more generally, 
 \[
s_m[\bv] (\by,\bx) := \frac{1}{m} \left( |\cD\bv(\by,\bx)|^m - 1 \right),
\]
where $m$ is a fixed positive constant in $[1,\infty)$, 
to accommodate different materials that exhibit different responses to external loads. The strain $s[{\bv}]$ corresponds to $m=1$. 
 For any deformation $\bv$ and any material point $\bx, \by$, we assume that $s_m[{\bv}](\by,\bx) \in (-\frac{1}{m}, \infty)$ 
 where a positive value indicates bond elongation and a negative value indicates bond compression. If $ s_m[{\bv}](\by,\bx)=0$, the bond length is unchanged by the deformation. 
 
Microelastic isotropic bond-based peridynamic materials are characterized  by a pairwise microelastic potential function $w= w(\bsxi,s)$ where the stored potential energy at a material point $\bx$ due to the deformation of the spring bond ${\bsxi=\by-\bx}$ is given by $w(\by-\bx,s_m[\bv](\by, \bx))$. 
Transitioning to macroelasticity, which concerns the body as a whole instead of interactions between pairs of points, at a point $\bx$, we define the {\it macroelastic energy density} functional
associated to the deformation $\bv$ as 
\begin{equation}\label{eq:totalelastic-density}
W(\bx, \bv) = \int_{\Omega}w(\by-\bx,s_m[\bv](\by,\bx) ) \, \rmd \by. 
\end{equation}
The total macroelastic energy, which represents the total stored energy accumulated through the deformation ${\bv}$ is thus given by 
 \begin{equation*}
 \int_{\Omega} W(\bx, \bv) \, \rmd \bx=\int_{\Omega}\int_{\Omega} w(\by-\bx,s_m[\bv](\by,\bx) )  \, \rmd \by \, \rmd \bx.
 \end{equation*} 
 The pairwise microelastic potential  function $w= w(\bsxi,s):\mathbb{R}^d\times (-1, \infty) \to \mathbb{R}$ is assumed to satisfy conditions that are natural in physical models. First, for any $\bsxi \in \bbR^d$, the map $w(\bsxi,\cdot)$ is nonnegative, locally Lipschitz, and twice differentiable near $s=0$.  
 Second, for any $\bsxi$, $w(\bsxi,s)=0$ if and only if $s=0$, and $\frac{\partial w}{\partial s} (\bsxi,0) = 0$.
That is, unchanged bond lengths do not contribute to the strain energy.
Finally, there exist positive constants $\alpha$ and $r$ and a nonnegative locally integrable function $\rho$ so that for any $\bsxi \in \bbR^d$, $w(\bsxi,s) \geq \alpha \rho(\bsxi) s^{2}$ for any $|s|<r$. 
This condition says that for small strain, the material obeys {\it Hooke's law}, where the restoring spring force is proportional to the elongation. Detailed conditions on $w$, as well as examples used in the literature, will be given in the next section.
A common feature of these microelastic energy functions is their global nonconvexity in the deformation.
That is, for each $\bx$, the integrand $W$ given by \eqref{eq:totalelastic-density} is nonconvex as a function of $|\cD\bv|$.

In this paper, we discuss two 
challenges that arise naturally in the study of
nonlinear nonconvex peridynamic models. The first is the rigorous justification of the derivation of the commonly-used linearized bond-based peridynamics model \cite{Silling2010, mengesha-du-2014-elasticity, alali2015peridynamics, mengesha2015VariationalLimit} 
obtained from nonlinear and nonconvex peridynamic models.
The second is the consistency of the peridynamic theory with the theory of classical continuum mechanics. This has been a central topic of study since the introduction of peridynamics as alternative model for multiscale modeling in solid mechanics,\cite{Silling2001,Silling2007,SL2008} (see also \cite{Bellido2015,bellido2020bond,mengesha2015VariationalLimit}).  To illustrate main ideas, we take specific values of $m$ in the initial mathematical description, and reintroduce a general value in the proofs in subsequent sections.

The derivation of the linearized bond-based peridynamics model takes the following path:
Formally, it is expected that for a naturally behaving elastic material, small external loads result in small displacements. Following the reasoning in \cite{DalMaso-linearization}, we write the deformation as $\bv(\bx) = \bx + \hat{\bu}(\bx)$, where $\hat{\bu}$ is the displacement field, $\hat{\bu} \equiv {\bf 0}$ corresponds to the equilibrium configuration, with corresponding deformation $\bi(\bx) = \bx + {\bf 0} = \bx$.
Now given a loading field $\bl$, let us consider a smoothly changing external loading field $\bl(\bx;\veps)$ with the property that $\bl(\bx;\veps) = \veps \bl(\bx) + O(\veps^2)$. Our formal understanding is that the resulting displacement field $\hat{\bu}(\bx; \veps)$ behaves similarly: $\hat{\bu}(\bx; \veps) = \veps \bu(\bx) + O(\veps^2)$ for a fixed underlying field $\bu$.
The corresponding deformation is $\bv_\veps= \bx + \veps \bu + O(\veps^2),$ and the first order approximation of the associated microstrain is $s[\bv_\veps](\by, \bx) = \left|{\cD\bi(\by,\bx)+ \veps\cD\bu(\by,\bx)}\right| -1$. 
The equilibrium configurations are then stationary points of the total elastic energy 
\begin{equation}\label{small-displacement-energy}
\int_{\Omega}\int_{\Omega} w(\by-\bx, \left|\cD\bi(\by,\bx) + \veps\cD\bu(\by, \bx) \right| -1) \, \rmd \by \, \rmd \bx - \veps^2\int_{\Omega} \bl(\bx)\cdot\bu(\bx) \, \rmd \bx\,. 
\end{equation}
After writing $\left|\cD\bi(\by,\bx) + \veps\cD\bu(\by, \bx) \right| -1
 = \frac{\left|\cD\bi(\by,\bx) + \veps\cD\bu(\by, \bx) \right|^{2} -1 }{\left|\cD\bi(\by,\bx) + \veps\cD\bu(\by, \bx) \right| +1}$, we have 
\[
\begin{split}
\left|\cD\bi(\by,\bx) + \veps\cD\bu(\by, \bx) \right| -1 
&=\veps \cD\bu(\by, \bx)\cdot \cD\bi(\by, \bx) \\
&+ \frac{|\veps\cD\bu|^2 + \veps \cD\bu(\by, \bx)\cdot \cD\bi(\by, \bx)(1 -|\cD\bi(\by,\bx) + \veps\cD\bu(\by, \bx)|) }{\left|\cD\bi(\by,\bx) + \veps\cD\bu(\by, \bx) \right| +1}\\
&=\veps \cD\bu(\by, \bx)\cdot \cD\bi(\by, \bx) + O(\veps^{2}|\cD\bu|^{2}(\by, \bx))\,.
\end{split}
\]
Now, if the vector field ${\bf u}$ is assumed to be Lipschitz, then using Taylor expansion on $w$ gives
\[
w(\by-\bx,\left|\cD\bi(\by,\bx) + \veps\cD\bu(\by, \bx) \right| -1) = \frac{\veps^2}{2}\frac{\partial^{2}w}{\partial s^2}(\by-\bx,0) (\cD\bu(\by, \bx)\cdot \cD\bi(\by, \bx))^{2} + O(\veps^2)\,.
\]
As a consequence,  
dividing the total energy in \eqref{small-displacement-energy} by $\veps^2$ and taking a limit in $\veps$, we obtain that 
\begin{equation}\label{lip-limit-energy}
\begin{split}
\lim_{\veps\to 0} \frac{1}{\veps^2 } &\left(\int_{\Omega}\int_{\Omega} w(\by-\bx, \left|\cD\bi(\by,\bx) + \veps\cD\bu(\by, \bx) \right| -1 )\, \rmd \by \, \rmd \bx - \veps^2\int_{\Omega} \bl(\bx)\cdot\bu(\bx)\, \rmd \bx \right)\\
 &= \frac{1}{2} \int_{\Omega}\int_{\Omega} \rho(|\by-\bx|)  (\cD\bu(\by, \bx)\cdot \cD\bi(\by, \bx))^{2} \, \rmd \by \, \rmd \bx - \int_{\Omega} \bl(\bx)\cdot\bu(\bx) \, \rmd \bx\,,
\end{split}
\end{equation}
where $\rho(\bsxi) = \frac{\partial^{2}w}{\partial s^2}(\bsxi,0)$ is the resulting interaction kernel.
The right-hand side of \eqref{lip-limit-energy} is usually taken as the total energy in linearized bond-based peridynamics models.  
We note that this convergence of the re-scaled sequence of potential energies is under the assumption that the field $\bu$ is Lipschitz, which excludes vector fields that may have discontinuities. In \cite[page 222]{C-N}, it is argued that weakening the condition for linearization is advantageous because it ``permits us to consider perturbations containing jump discontinuities in ${\bu}$.'' Indeed, the ability to model discontinuous deformation fields is a central attractive feature of peridynamics models.
In this regard, our first main result justifies the linearized bond-based peridynamics model for vector fields that make the right-hand side of \eqref{lip-limit-energy} finite, which is a class of vector fields that is bigger than that of Lipschitz vector fields. Most importantly, depending on the singularity of the interaction kernel $\rho$, this class may contain discontinuous vector fields. We will state the result precisely in the next section but broadly speaking we will prove that the sequence of functionals
\begin{equation}\label{nl-total-energy}
E_{\veps}(\bu)= \frac{1}{\veps^2 }\int_{\Omega}\int_{\Omega} w(\left|\cD\bi(\by,\bx) + \veps\cD\bu(\by, \bx) \right| -1, |\by-\bx|)\, \rmd \by \, \rmd \bx - \int_{\Omega} \bl(\bx)\cdot\bu(\bx)\, \rmd \bx 
\end{equation}
will
\textit{$\Gamma$-converge} with respect to the strong $L^{2}$ topology to 
\begin{equation}\label{quad-energy}
 E_{0}(\bu) = \frac{1}{2} \int_{\Omega}\int_{\Omega} \rho(|\by-\bx|)  (\cD\bu(\by, \bx)\cdot \cD\bi(\by, \bx))^{2} \, \rmd \by \, \rmd \bx - \int_{\Omega} \bl(\bx)\cdot\bu(\bx)\, \rmd \bx. 
\end{equation}
While it remains an open question, the nature of this mode of convergence is that, under some additional condition, minimizers of \eqref{nl-total-energy} subject to some volumetric conditions converge to a minimizer of the limiting energy functional \eqref{quad-energy}. 

The second issue we discuss in this paper is the consistency of the peridynamics theory with the classical continuum mechanics theory. This has been a central topic of study since the introduction of peridynamics as alternative model for multiscale modeling in solid mechanics \cite{SL2008,Silling2001,Silling2007} (see also \cite{Bellido2015,bellido2020bond, mengesha2015VariationalLimit}). 
The underlying assumption in the peridynamics model is that there are direct interactions between points that are at a finite distance from each other. As described in \cite{SL2008}, the maximum interaction distance introduces a length scale for the material model, which is called the \textit{horizon}. A common justification for the consistency of the peridynamic model with the classical model is to derive and analyze the form of the resulting strain energy functional as this
horizon vanishes. For linearized peridynamics, its consistency with linearized elasticity has been established with various level of rigor, including proofs via variational convergence. Indeed, associated to a given interaction kernel $\rho\in L^{1}_{loc}(\mathbb{R}^{d})$ that is radial with $\rho(|\bsxi|)=0$ for $|\bsxi|>1$,  we can introduce a sequence of kernels $\rho_{\delta}(|\bsxi|) = \frac{1}{\delta^{d}} \rho(\frac{|\bsxi|}{\delta})$, where in this case $\delta$ represents the horizon. As shown in \cite{mengesha2015VariationalLimit}, 
the sequence of quadratic potential energy functionals \eqref{quad-energy} associated to $\rho_\delta$ have a $\Gamma$-limit with respect to the strong $L^2$-topology, and this $\Gamma$-limit $E_{loc}$ is a gradient energy given by
\[
E_{loc}(\bu) = \mu\int_{\Omega}2|\text{Sym}(\nabla \bu)|^2  + (\div {\bf u})^2 \, \rmd \bx - \int_{\Omega}\bl(\bx)\cdot \bu(\bx) \, \rmd \bx\,.
\]
The functional $E_{loc}$ is the well known Navier potential energy corresponding to a constant Poisson ratio ${1\over 4}$ and bulk modulus $\mu$ calculated from $\rho$. The consistency of the more general state-based linearized peridynamic model to that of the linearize elasticity  model for arbitrary Poisson ratio is additionally established in \cite{mengesha2015VariationalLimit}. 
Note in particular that the convergence preserves the class of displacements that correspond to unstrained configurations. Indeed, the same class of infinitesimal rigid motions $\bu=\mathbb{Q}\bx + {\bf b}$ with $\mathbb{Q}^{T} = -\mathbb{Q}$ solve the differential equation
$
\text{Sym}(\nabla \bu) = { \bf 0}$ in $\Omega
$
as well as the difference equation $\cD\bu(\by, \bx)\cdot \cD\bi(\by, \bx) = 0$ for almost all $\bx, \by \in \Omega$.

The situation in the case of nonlinear peridynamics is quite different. To demonstrate this, let us consider the sequence of microelastic energy densities 
\[
w_\delta(\bsxi, s) = \rho_\delta(|\bsxi |)(s^2-1)^2
\]
introduced in \cite{Silling2001} for an isotropic homogeneous microelastic peridynamic material, where $\rho_\delta$ is as above.
The corresponding macroscopic strain energy associated with the deformation $\bv$ is
\begin{equation}\label{nonlinear-StV}
\int_\Omega\int_{\Omega} \rho_{\delta}(\by-\bx)(|{\mathcal{D} \bv}(\by, \bx)|^2 - 1)^2 \, \rmd \by \, \rmd \bx\,,
\end{equation}
which vanishes when the deformation $\bv$ is distance-preserving, i.e., 
\begin{equation}\label{eq:Intro:DistPreserve}
    |\bv(\bx)-\bv(\by)| = |\bx-\by| \text{ for all } \bx, \by \in \Omega\,.
\end{equation} 
It is well-known that such maps are necessarily affine maps whose gradient is a rotation (see Theorem \ref{lma:MeasFxnsAreRigid} for a proof).
If we fix a twice-continuously differentiable map $\bv$ and let $\delta$ tend to $0$ in the sequence of macroscopic strain energies, we obtain
\begin{equation}\label{eq:Intro:LocalLimit}
    \lim_{\delta \to 0} \int_\Omega\int_{\Omega} \rho_{\delta}(\by-\bx)(|{\mathcal{D} \bv}(\by, \bx)|^2 - 1)^2 \, \rmd \by \, \rmd \bx = \int_{\Omega} \bar{w}(\nabla \bv (\bx))\, \rmd \bx
\end{equation}
where for any $d \times d$ matrix, $\mathbb{F}$, $\bar{w}(\mathbb{F}) = 2\mu\|\bbF^T \bbF-\bbI\|^2 + \mu(\|\bbF\|^2 -d)^2$ for some $\mu = \mu(d)$. This can be obtained via direct calculation, see also \cite{Bellido2015,ponce2004new}. The energy density on the right-hand side is the St.\ Venant-Kirchhoff strain energy density function of a hyperelastic material with bulk modulus $\mu$. The twice-continuously differentiable deformations corresponding to unstrained configurations for the St.\ Venant-Kirchhoff material are those $\bv : \Omega \to \bbR^d$ such that 
\begin{equation}\label{eq:Intro:DiffInc}
    \nabla \bv(\bx)^{T} \nabla \bv(\bx)=\bbI \quad \text{ for all } \bx \in \Omega\,.
\end{equation}
Any smooth vector field $\bv$ that satisfies \eqref{eq:Intro:DiffInc} has gradient matrix $\nabla \bv(\bx)$ equal to a constant orthogonal matrix for all $\bx\in \Omega$, and so necessarily $\bv$ is affine; see Theorem \ref{thm:LocalRigiditySmoothFxns}. So if the class of deformations is restricted to smooth functions, the zero sets of \eqref{nonlinear-StV} and \eqref{eq:Intro:LocalLimit} coincide.

However, the presence of external body forces and general boundary conditions can introduce discontinuities in the deformation gradient and its higher derivatives. Indeed, since the peridynamic functionals depend only on first-order difference quotients,
it is natural to consider the localization limit in a topology permitting discontinuous gradients, such as the uniform topology of continuous functions.
These discontinuities can interact with the non-convexity of the functionals to produce effective behavior of minimizers that is quite different than the above uniform limit would suggest. Let us illustrate with the following one-dimensional example. Take $\Omega = (0,1)$, and for a number $N \in \{1,2,\ldots\}$ define $v_N : [0,1] \to \bbR$ as the following function that has a ``sawtooth'' profile:
\begin{equation*}
    v_N(x) :=
    \begin{cases}
        x - \frac{2k}{2N}\,, &\text{ for } x \in [\frac{2k}{2N}, \frac{2k+1}{2N}) \\
        -x + \frac{2k+2}{2N}\,, &\text{ for } x \in [\frac{2k+1}{2N}, \frac{2k+2}{2N}) 
    \end{cases}
    \qquad \text{ for } k \in \{0,1,\ldots,N-1\}\,.
\end{equation*}
For each $N$, $v_N$ is a Lipschitz continuous function, with weak derivative defined almost everywhere taking the value of either $-1$ or $1$.
Then, for the kernel $\rho$ defined as $\rho(z) = \frac{1}{2}$ if $|z| < 1$ and $0$ otherwise, it is straightforward to calculate
\begin{equation*}
    \int_0^1 \int_{ \{ |z| < \delta \} } \rho_{\delta}(y-x)(|{\mathcal{D} v_N}(y, x)|^2 - 1)^2 \, \rmd y \, \rmd x = \frac{8}{15} N \delta\,. 
\end{equation*}
For the choice of  $\delta = \delta(N) = \frac{1}{N^2}$ (or similar), the sequence of energies tends to $0$ as $N \to \infty$, i.e. we expect $v_N$ to be a minimizing sequence of the localized limit. However, $v_N \to 0$ uniformly in $[0,1]$, which clearly does not minimize $ \int_{0}^1 \bar{w}(v'(x))\, \rmd x$. Thus, this local functional does not capture the effective behavior of the nonlocal functionals in the localization limit taken in this topology.

More generally, the relations \eqref{eq:Intro:DistPreserve} and \eqref{eq:Intro:DiffInc} interact with uniform convergence in different ways. On one hand, if a sequence of Lipschitz continuous deformations $\{ \bv_k \}$ satisfying \eqref{eq:Intro:DistPreserve} has a uniform limit $\bv$, that continuous deformation $\bv$ must necessarily satisfy \eqref{eq:Intro:DistPreserve}; see Theorem \ref{thm:ReshetnyakVersion1} for a more general statement. That is, the relation \eqref{eq:Intro:DistPreserve} is \textit{rigid}. On the other hand, the relation \eqref{eq:Intro:DiffInc} is \textit{not} rigid, since there exist sequences of Lipschitz continuous deformations $\{ \bv_k \}$ satisfying \eqref{eq:Intro:DiffInc} that converge uniformly to $\bv = {\bf 0}$, which does not satisfy \eqref{eq:Intro:DiffInc}.

To summarize, in the regime of variational convergence the behavior of the nonlocal functionals is not retained in the localization limit.
As a special case of the result we obtain in this paper, it turns out that the $\Gamma$-limit of \eqref{nonlinear-StV} has an integral representation with quasiconvex density $w_\infty$ of the form $\int_{\Omega} w_{\infty}(\nabla \bv) \, \rmd \bx$ for $\bv \in W^{1, 4}(\Omega;\mathbb{R}^{d})$, 
an integral functional that depends only on 
the gradient of the vector field. Moreover, we will show that 
\[
w_{\infty}(\bbF) = 0 \,\text{ if and only if }\, \bbF^{T}\bbF \leq \bbI\text{ in the sense of quadratic forms}. 
\]
This implies that the zero set of the limiting functional consists of Sobolev maps whose gradient satisfies $\nabla \bv^{T}(\bx)\nabla \bv(\bx)\leq \bbI$ for almost all $\bx\in \Omega$, which is a larger class than the zero set of the St.\ Venant-Kirchhoff strain energy functional. This class of the zero sets also makes it clear that no classical hyperelastic material model can be obtained as a $\Gamma$-limit of the sequence of nonlocal energy functionals of the form \eqref{nonlinear-StV}.

The paper is organized as follows: In Section \ref{sec:SMR} we define notation and precisely state the main results of the paper. Section \ref{sec:Lin} contains the $\Gamma$-convergence result for the linearization of the nonlocal functionals.
In Section \ref{sec:Loc} we prove the integral representation result for the $\Gamma$-limit of the nonlocal functionals in the localization limit.
We then derive further properties of this integral representation in Section \ref{sec:LocProp}.

\section{Statement of main results}\label{sec:SMR}
Before we state the main results of the paper, let us fix notation as well as assumptions. 
The open Euclidean ball centered at $\bx_0 \in \bbR^d$ with radius $r$ is denoted $B(\bx_0,r)$.
The set of $d \times d$ matrices with real entries is denoted $\bbR^{d \times d}$.
We denote the transpose of a matrix $\bbF \in \bbR^{d \times d}$ by $\bbF^T$. The set $\cO(d) \subset \bbR^{d \times d}$ consists of all $\bbF$ such that $\bbF^T \bbF = \bbI$, where $\bbI$ denotes the $d \times d$ identity matrix.
We denote $d$-dimensional Lebesgue measure as $\rmd \bx$, and denote $d-1$-dimensional surface measure as $\rmd \sigma$. For a measure $\nu$ and a $\nu$-measurable function $f$, the integral average is denoted by $\fint_A f \, \rmd \nu = \frac{1}{\nu(A)} \int_A f \, \rmd \nu$. 
For $1 \leq p \leq \infty$, standard Lebesgue and Sobolev spaces for general vector fields $\bu : \bbR^d \supset \Omega \to \bbR^N$ are denoted $L^p(\Omega;\bbR^N)$ and $W^{1,p}(\Omega;\bbR^N)$ respectively, with the codomain omitted when $N = 1$.

We assume that the microelastic potential $w : \bbR^d \times [-1,\infty) \to [0,\infty)$ satisfies the following:
\begin{enumerate}
    \item[i)] $w(\bsxi,s) = k(\bsxi) \Psi(\bsxi,s)$, where $\Psi$ is a Carath\'eodory  function with $s \mapsto \Psi(\bsxi,s)$ locally Lipschitz on $[-1,\infty)$.
    \item[ii)] There exists $\delta_0>0$ such that for each $\bsxi$, $\Psi(\bsxi,\cdot) \in C^2([0,\delta_0])$,  $\Psi(\bsxi,0) = 0$,$\frac{\p \Psi}{\p s}(\bsxi,0) = 0$, $ \frac{\p^2 \Psi}{\p s^2}(\bsxi,0) >0$
and 
    \begin{equation*}
        \inf_{|s| \geq s_0} \inf_{\xi \in \bbR^d} \Psi(\bsxi,s) > 0 \text{ for all } s_0 > 0\,.
    \end{equation*}
    \item[iii)]
    There exist constants $c_1 > 0$ and $\delta_0 > 0$ such that for all $\bsxi$
    \begin{equation*}
        k(\bsxi)\Psi(\bsxi,s) \geq c_1 k(\bsxi)s^2 \text{ for } |s| < \delta_0\,.
    \end{equation*}
    \item[iv)] There exists $c_2 >0$ such that
    \begin{equation*}
       k(\bsxi) \left| \frac{\p^2 \Psi}{\p s^2}(\bsxi,s)  \right| \leq c_2 k(\bsxi) \quad \text{ for } |s| < \delta_0\, \text{and $\bsxi\in \bbR^d$}.
    \end{equation*}
\end{enumerate}
It follows from the above assumptions that that for some $c>0,$ and all $\bsxi\in \bbR^d$, 
\begin{equation*}
     k(\bsxi)  \frac{\p^2 \Psi}{\p s^2}(\bsxi,0)s^2   \geq c  k(\bsxi)  s^2 \text{ for all } s \in \bbR\,.
\end{equation*}
Define $\rho :\bbR^{d}\to [0, \infty)$ by 
\[
\rho(\bsxi) =  k(\bsxi)  \frac{\p^2 \Psi}{\p s^2}(\bsxi,0)
\] and assume it to be in $L^{1}_{loc}(\bbR^d)$. The kernel will be used to define  
 the function space
\begin{equation}\label{dot-space}
\mathcal{X}_\rho(\Omega) = \left\{\bu\in L^{2}(\Omega;\bbR^d): \int_{\Omega}\int_{\Omega} \rho(\by-\bx)(\mathcal{D}\bu(\by,\bx)\cdot \mathcal{D}\bi (\by, \bx))^2 \, \rmd \by \, \rmd \bx < \infty\right\}. 
\end{equation} 
We note that for any $\bl$ and $\bu\in L^{2}(\Omega,\bbR^d)$ the functional $ E_0(\bu)$ defined in \eqref{quad-energy}
    is finite if and only if $\bu\in \mathcal{X}_\rho(\Omega)$. We are now ready to state the $\Gamma$-convergence result for the linearization regime.
\begin{theorem}\label{thm:Linearization}
Suppose the microelastic energy density function $w(\bsxi, s)$ satisfies the above conditions \textrm{i})-\textrm{iv}). 
    For a given external load $\bl \in L^2(\Omega;\bbR^d)$,
    define the functional
    $\overline{E}_{\veps} : L^2(\Omega;\bbR^d) \to \overline{\bbR}$ by
    \begin{equation*}
        \overline{E}_{\veps}(\bu) = 
        \begin{cases}
            E_\veps(\bu)\,, & \quad \text{if $E_\veps(\bu)<\infty$},\\
            +\infty\,, & \quad \text{otherwise.}
        \end{cases}
    \end{equation*}
    where $E_\veps$ is the sequence of functionals introduced in \eqref{nl-total-energy}. 
    Then, with respect to the strong topology on $L^2(\Omega)$, we have 
    \begin{equation*}
        \GammalimL2_{\veps \to 0^+}
        \overline{E}_{\veps}(\bu) = \overline{E}_0(\bu) := 
        \begin{cases}
            E_0(\bu)\,, & \quad \bu \in \mathcal{X}_\rho(\Omega)\,, \\
            +\infty\,, & \quad \bu \in L^2(\Omega;\bbR^d) \setminus \mathcal{X}_\rho(\Omega).
        \end{cases}
    \end{equation*}
\end{theorem}
Throughout the literature on bond-based peridynamics, a number of microelastic energy density functions are proposed that satisfy some of the above conditions. One example is a peridynamic formulation of a constitutive law for microelastic brittle materials  (MBM), in which the microelastic force density ${\partial w\over \partial s}(\bsxi,s)$ is zero if the strain $s$ exceeds a given threshold $s_0$. In this case, the  density is taken to be $w(\bsxi,s) = k(\bsxi)\psi(s)$ where 
$\psi(s) = 
 c {s^{2}\over 2}$ if $s\leq s_0$ and 
$ c {s_0^{2}\over 2}$ if $s>s_0$.
One can think of the spring connecting the material particles as linear for small strain, but the spring breaks when the strain exceeds $s_0$, see \cite[Section 14]{Silling2001}. 
A second example is a modified MBM model, which introduces a strain zone where the force density weakens and eventually vanishes instead of vanishing immediately after vanishing a threshold strain, see \cite{C-N,Q-T-T} for more. 
 In \cite{Lipton-Brittle, Lipton-Cohesive}, still another version is studied: 
 \begin{equation*}
 w(\bsxi,s) = k(\bsxi) f(\bsxi s^2)
 \end{equation*}
where $f:[0, \infty)\to \mathbb{R}$ is a nonnegative, smooth and convex-concave function with the properties
 \[
f(0)=0, \quad \lim_{r\to 0+} {f(r)\over r} = f'(0) >0, \quad\lim_{r\to \infty} f(r) = f_{\infty} < \infty.
 \]
 As discussed in \cite{Lipton-Cohesive}, this choice of energy density gives a microelastic force density function that is linear for small strains and that weakens indefinitely after a certain strain threshold. For a fixed $\bsxi$, all of the above mentioned microelastic energy densities are bounded as a function of $s$. 
 In \cite{Silling2001}, Silling argued that classical energy density functions for some hyperelastic materials can be computed from  given microelastic potential functions such as 
 \begin{equation*}
 w(\bsxi,s) = k(|\bsxi|)(s^2-1)^2\,.
 \end{equation*}
One can also construct other examples such as $w(\bsxi,s) = k(\bsxi)(g(s + 1)-1)^2$ where $g$ is differentiable, $g(1)=1$ and $g'(1)\neq 0$. 
In connection with modeling phase transition via peridynamics, authors in \cite{Kaushik} and recently in\cite{aguiar} have also studied microelastic potentials of the form $w(\bsxi,s) = \rho(\bsxi)\min\{s^2, (s-s_0)^{2}\}$ for some nonnegative $s_0$.

Our second result holds for a specialized class of microelastic potentials. Here we assume that $w= w(\bsxi,s): \bbR^d \times (-1, \infty) \to \mathbb{R}$ is of the form
 \begin{equation*}
 	w_n(\bsxi,s) = \rho_n(\bsxi) \Phi(|s|) 
 \end{equation*}
 where $\Phi 
 		: [0,\infty) \to [0,\infty)$ is a nondecreasing and convex function with $p$-growth:
 \begin{equation}\label{eq:PDEnergy:ConditionsOnG}
 	\begin{gathered}
 		C_0(|a|^p - 1) \leq \Phi(a) \leq C_1(1+|a|^p)\,, \qquad \text{ for all } a \in [0,\infty)\,, \\
 		\Phi(a) = 0 \text{ if and only if } a =0\,,
 	\end{gathered}
 \end{equation}
where $C_0$, $C_1$ are given positive constants and $p \in (1,\infty)$. The sequence of  nonnegative radial kernels $\{ \rho_n \}_{n \in \bbN}$ are assumed to converge to the Dirac-Delta measure at 0, i.e.
\begin{gather}
	\int_{\bbR^d} \rho_n(\bx) \, \rmd \bx = 1\,, \qquad \lim_{n \to \infty} \intdm{\bbR^d \setminus B(\boldsymbol{0},\delta)}{\rho_n (\bsxi)}{\bsxi} = 0 \qquad \forall \, \delta > 0\,.
 \label{eq:AssumptionsOnKernels1} \tag{A}
\end{gather}
For each $n$, we now consider the sequence of functionals
\begin{equation}\label{p-growing-nl-energy}
F_{n}(\bv) = \iintdm{\Omega}{\Omega}{\rho_n(\bx-\by) \Phi\left( \left|\frac{|\bv(\bx)-\bv(\by)|}{|\bx-\by|} - 1 \right|\right) }{\by}{\bx}.
\end{equation}
From the growth conditions \eqref{eq:PDEnergy:ConditionsOnG} for $\Phi$, we see that for each $n$, $F_{n}(\bv) < \infty$ if and only if $\bv\in \mathfrak{W}^{\rho_n,p}(\Omega;\bbR^d)$ where for any $\rho\in L^{1}_{loc}(\mathbb{R}^d)$, the function space $\mathfrak{W}^{\rho,p}(\Omega;\bbR^d)$ is defined as 
\[
\mathfrak{W}^{\rho,p}(\Omega;\bbR^d) :=\left\{\bv\in L^{p}(\Omega;\mathbb{R}^d): [\bv]_{\mathfrak{W}^{\rho,p}(\Omega)}^{p}:=\iintdm{\Omega}{\Omega}{\rho(\bx-\by)  \frac{|\bv(\bx)-\bv(\by)|^p}{|\bx-\by|^p} }{\by}{\bx} < \infty\right\}. 
\]
This space has a natural norm $\|\cdot\|_{\mathfrak{W}^{\rho,p}(\Omega)}  := (\|\cdot\|^{p}_{L^{p}(\Omega)} + [\cdot]_{\mathfrak{W}^{\rho,p}(\Omega)}^{p})^{\frac{1}{p}}$ and is a reflexive Banach space with respect to this norm. 
Moreover, a well known result of Bourgain, Brezis, and Mironescu \cite{BBM} (see also \cite{ponce2004new}) states that the sequence of spaces $\mathfrak{W}^{\rho_n,p}(\Omega;\bbR^d)$ can be used to characterize the Sobolev space $W^{1,p}(\Omega;\mathbb{R}^{d})$ in the following sense: for any $\bv\in W^{1,p}(\Omega;\mathbb{R}^d)$, $\liminf_{n\to \infty}[\bv]^{p}_{\mathfrak{W}^{\rho_n,p}(\Omega)} = \int_{\Omega}\fint_{\mathbb{S}^{d-1}}|\nabla \bv(\bx)\bseta|^{p} \, \rmd \sigma(\bseta) \, \rmd \bx$, which is a seminorm equivalent to the $W^{1,p}(\Omega;\mathbb{R}^d)$-seminorm, and conversely, if $\bv\in L^{p}(\Omega;\mathbb{R}^d)$ and $\liminf_{n\to \infty}[\bv]^{p}_{\frak{W}^{\rho_n,p}(\Omega)}<\infty$, then $\bv\in W^{1,p}(\Omega;\mathbb{R}^d)$. 

The variational limit of the nonlinear peridynamic functionals \eqref{p-growing-nl-energy} in the localization limit is described in the following theorem:
\begin{theorem}\label{thm:Intro:GammaConvergence}
	Define the functionals $\overline{F}_n : L^p(\Omega;\bbR^d) \to \overline{\bbR}$ to be extensions of $F_n$ in \eqref{p-growing-nl-energy} to $L^p(\Omega;\bbR^d)$; that is,
	\begin{equation*}
		\overline{F}_n(\bv) := \begin{cases}
			F_n(\bv)\,, & \bv \in \frak{W}^{\rho_n,p}(\Omega;\bbR^d) \\
			+\infty\,, & \text{ otherwise. }
		\end{cases}
	\end{equation*}
	Then there exists a quasiconvex Carath\'eodory function $f_\infty : \bbR^{d \times d} \to \bbR$ satisfying the growth condition
	\begin{equation*}
		C(|\bbF|^p - 1) \leq f_\infty(\bbF) \leq C (|\bbF|^p + 1)
	\end{equation*}
	such that $\overline{F}_n$ $\Gamma$-converges in the strong topology on $L^p(\Omega;\bbR^d)$ to $\overline{F}_\infty$, where the 
 limiting functional $\overline{F}_{\infty} : L^p(\Omega;\bbR^d) \to \overline{\bbR}$ is defined as
	\begin{equation*}
		\overline{F}_{\infty}(\bv) = \begin{cases}
			\intdm{\Omega}{f_\infty(\grad \bv)}{\bx}\,, & \bv \in W^{1,p}(\Omega;\bbR^d)\\
			+\infty\,, & \text{ otherwise. }
		\end{cases}
	\end{equation*}
	Moreover, the effective strain energy density $f_\infty$ is frame-indifferent, i.e.\ $f_\infty(\bbU \bbF) = f_\infty(\bbF)$ for all $\bbU \in \cO(d)$, and $f_\infty(\bbF) = 0$ if and only if $\bbF^T \bbF \leq \bbI$ in the sense of quadratic forms.
\end{theorem}
Some remarks are in order. First, as discussed in the introduction, even though the energies $F_n$ vanish only for affine functions with gradient in $\cO(d)$, the limiting functional $\overline{F}_{\infty}$ vanishes for a much wider class of nonaffine functions. 
As a direct consequence, one notes that \textit{no stored energy for a classical hyperelastic material can be recovered as a $\Gamma$-limit of bond-based peridynamic functionals,} since the zero-set of the limiting functional $\overline{F}_{\infty}$ is always strictly larger than that of any stored energy functional of classical theory.  This observation is a constructive counterpart to 
a negative result in \cite{bellido2020bond} in which some functionals with polyconvex integrand are shown not to be obtainable via $\Gamma$-convergence of bond-based peridynamic functionals.

Second, the integral representation of $\Gamma$-limits of nonlocal functionals are available in the literature. For example,
Theorem \ref{thm:LocalGammaConvergence} is closely related to the results contained in \cite{alicandro2004general} for discrete energies. Our work is in fact inspired by the result and the approach used in \cite{alicandro2004general}. Similar results are also proved in \cite{alicandro2020variational} where energy densities of a very general form with symmetric kernels are studied.  However, the sequences of kernels considered there are exclusively of convolution type and have weaker singularity conditions than \eqref{eq:AssumptionsOnKernels1}, namely that $|\bsxi|^{-p}\rho(\bsxi) \in L^1(\bbR^d)$, and that $\rho_n(\bsxi) = n^d \rho(n \bsxi)$. Thus, for the theory developed in \cite{alicandro2020variational} the space $\frak{W}^{\rho_n,p}(\Omega;\bbR^d)$ is always $L^p(\Omega;\bbR^d)$, whereas in the present work the space may be strictly smaller; see \cite[Remark 2.5]{mengesha2015VariationalLimit}. In addition, the present work covers localizing kernels such as $\rho_s(|\bsxi|) = C(1-s)|\bsxi|^{d+sp-p} \chi_{B(\boldsymbol{0},1)}(\bsxi)$ for $s \in (0,1)$ and some normalizing constant $C$. The recent work \cite{Braides-DalMaso} has a closely-related result, in which the $\Gamma$-convergence in the weak topology on Sobolev spaces of a class of integral functionals expressed as a sum of a local and a non-local terms are studied. The sequence of microenergy densities we study indeed satisfy the growth and integrability conditions of the term that comprises the nonlocal part of \cite{Braides-DalMaso}. However, the additional convexity structure on $w$ with respect to the strain $s$ 
is important for our work as it allows us to obtain estimates for the limiting integral representation. Integral representations of nonlocal functionals are also obtained in \cite{ponce2004new} and \cite{Bellido2015}. In both of those works, sufficient conditions are provided that force the $\Gamma$-limit to be the integral of a quasiconvexification of the point-wise limit, i.e. 
$
\bar{w}(\mathbb{F}) = \fint_{\mathbb{S}^{d-1}} \Phi(|\mathbb{F} \bsnu| )d\sigma(\bsnu). 
$
The conditions effectively amount to requiring the convexity of $\Phi$ in $\mathbb{F}$, which is 
outside the setting of the present work. In fact, our work partially addresses one of the open problems posed in \cite{ponce2004new} that -- even in the absence of the sufficient condition -- an integral representation of the $\Gamma$-limit is possible.  Using the same techniques in \cite{ponce2004new}, upper and lower bounds for the $\Gamma$-limit can be obtained and we will present an example that shows these bounds may not be equal.

\section{Justification of linearization}\label{sec:Lin}
In this section we will present the proof of Theorem \ref{thm:Linearization} which states the proper justification of  bond-based peridynamics as a linearized model of nonlinear peridynamics.  The proof we present below also covers energy functionals whose corresponding microelastic energy density is a function of the more general nonlocal nonlinear strain $s_m[\bv] (\by,\bx) = \frac{1}{m} \left( |\cD\bv(\by,\bx)|^m - 1 \right)\,$ for $m\geq 1$. To that end, let us begin with a preliminary result related to the linearization of $s_m$.
\begin{lemma}\label{lma:NonlinearStrain:Taylor}
Let $m\geq 1$. For a given $(\bsnu, \bszeta)\in \bbS^{d-1}\times \bbR^{d}$ and $\veps > 0$, define the function 
\[
\wt{s}_m(\bsnu,\veps\bszeta) := \frac{1}{m} \left( |\bsnu + \veps \bszeta|^m - 1\right). 
\]
Then for any $\veps>0$, there exists a function $\psi$ such that 
\[\wt{s}_m(\bsnu,\veps\bszeta) = \veps \bsnu\cdot \bszeta + \veps^2 \psi(\veps, \bsnu,\bszeta).\]
Moreover, if $\veps_0>0$, and $ M>0$ are given  then there exists a constant $C>0$ that depends only on $d, m, \veps_0,$ and $M$ so that for all $(\veps, \bsnu, \bszeta)\in [0, \veps_0]\times \bbS^{d-1}\times B(\boldsymbol{0},M)$
\[
|\psi(\veps, \bsnu,\bszeta)| \leq C.
\]
\end{lemma}
\begin{proof}
Let us first assume that $m\geq 2$. For a fixed $(\bsnu, \bszeta)\in \bbS^{d-1}\times \bbR^{d}$, using Taylor expansion gives 
\begin{equation}\label{eq:Linearization:Liminf:Pf1}
    \wt{s}_m(\bsnu,\veps \bszeta) = \frac{1}{m} \left(|\bsnu + \veps \bszeta|^m - 1\right) = \veps \bszeta \cdot \bsnu + \frac{\veps^2}{2} \frac{\p^2 \wt{s}_m}{\p \bszeta^2} (\bsnu,t \veps \bszeta) [\bszeta, \bszeta]
\end{equation}
for $t=t(\bsnu, \bszeta, \veps) \in (0,1)$, where
\begin{equation}
\label{eq:Linearization:Liminf:Pf2}
    \frac{\p^2 \wt{s}_m}{\p \bszeta^2 }(\bsnu,\bszeta) =  |\bszeta + \bsnu|^{m-2} \left( \bbI + (m-2) \frac{(\bszeta + \bsnu) \otimes (\bszeta + \bsnu)}{|\bszeta + \bsnu|^2} \right)\,.
\end{equation}
The first part of the lemma follows now by taking $\psi(\veps, \bsnu,\bszeta)=\frac{\p^2 \wt{s}_m}{\p \bszeta^2} (\bsnu,t \veps \bszeta) [\bszeta, \bszeta].$
Notice from \eqref{eq:Linearization:Liminf:Pf2} that for any unit vector $\bsnu$ and any vector $\bszeta$ 
\[
\left|\frac{\p^2 \wt{s}_m}{\p \bszeta^2} (\bsnu,t \veps \bszeta) [\bszeta, \bszeta]\right| \leq C(d,m)|\bsnu+ t\veps \bszeta|^{m-2}. 
\]
It then follows that if $|\bszeta|\leq M,$ and $0\leq\veps \leq \veps_0$, then $|\bsnu + t\veps \bszeta|^{m-2} \leq C(d, m, \veps_0, M)$. That completes the proof in the event that $m\geq2$. 

For the case $1\leq m<2$, we write, after multiplying and dividing by $|\bsnu + \veps \bszeta|^{m} +1$
\[
 \wt{s}_m(\bsnu,\veps \bszeta) = \frac{1}{m} \frac{|\bsnu + \veps \bszeta|^{2m} - 1}{|\bsnu + \veps \bszeta|^{m} + 1} = {1\over {2m}} (|\bsnu + \veps \bszeta|^{2m} - 1) {2\over |\bsnu + \veps \bszeta|^{m} + 1}
\]
Now since $2m \geq 2$, we may apply the first case to write as
\[
\wt{s}_m(\bsnu,\veps \bszeta) = (\veps \bsnu\cdot \bszeta + \veps^2 \wt{\psi}(\veps, \bsnu,\bszeta)) {2\over |\bsnu + \veps \bszeta|^{m} + 1}. 
\]
Using Taylor's expansion, we can also write
\[{2\over |\bsnu + \veps \bszeta|^{m} + 1} = 1 + \veps \phi(\veps, \bsnu, \bszeta), \quad \text{where $|\phi(\veps, \bsnu, \bszeta)| \leq 2m|\bsnu+ t\veps \bszeta|^{m-1}|\bszeta|$ and $t\in (0, 1)$.}\]
 Combining the above equalities we conclude the proof of the lemma. 
\end{proof}

\begin{corollary}\label{pointwise-limit-density}
    Assume that $w:\bbR^{d}\times (-1, \infty)\to [0, \infty)$ satisfies all the conditions stated in Theorem \ref{thm:Linearization}. Then, given a bounded subset $K\subset \bbR^d$, sufficiently small $\veps_0 >0$, and $\bsxi\neq 0$, we have
    \[
    \sup_{0<\veps\leq \veps_0} \sup_{\bszeta \in K} \left| {1\over \veps^2} w(\bsxi, \wt{s}_m({\bsxi/|\bsxi|},\veps \bszeta)) \right| \leq C \, k(\bsxi)\,,
    \]
    \[
    \lim_{\veps \to 0} {1\over \veps^2} w(\bsxi, \wt{s}_m({\bsxi/|\bsxi|},\veps \bszeta)) = k(\bsxi)\frac{\p^2 \Psi}{\p s^2}( \bsxi, 0) \left(\bszeta \cdot\frac{ \bsxi}{|\bsxi|}\right)^2\,, \quad \forall \bszeta \in K \text{ and } \veps \in (0,\veps_0)
    \]
\end{corollary}
\begin{proof}
    The proof follows directly from Lemma \ref{lma:NonlinearStrain:Taylor}, the assumptions on $w$, and the Taylor expansion 
    \begin{equation}\label{loc-limit-fixed-xieta}
\frac{1}{\veps^2} w(\bsxi, \wt{s}_m(\bsxi/|\bsxi|,\veps \bszeta)) = \frac{1}{\veps^2} k(\bsxi)\Psi(\bsxi, \wt{s}_m(\bsxi/|\bsxi|,\veps \bszeta) )=k(\bsxi)\frac{\p^2 \Psi}{\p s^2}( \bsxi, \veps^*) \left(\bszeta \cdot\frac{ \bsxi}{|\bsxi|} + O(\veps)\right)^2
\end{equation}
for some $\veps^* > 0$ that depends on $\veps$, $\bsxi$, and $\bszeta$. Moreover, for a fixed $\bsxi \neq 0$, $\displaystyle \lim_{\veps \to 0} \sup_{\bszeta \in K } \veps^* = 0$. 
\end{proof}
We also need the following result on the density of smooth functions in the nonlocal function space $\mathcal{X}_\rho(\Omega)$. The proof follows an argument that establishes the corresponding result for Sobolev spaces, 
see for instance \cite{Evans}. We include the proof here to ensure its validity on a space equipped with a ``non-standard'' norm.

\begin{theorem}\label{global-boundary-approximation}
Assume that $\rho$ satisfies \eqref{eq:AssumptionsOnKernels1} and that $\Omega$ is a bounded Lipschitz domain. Then $C^{\infty}(\overline{\Omega}; \mathbb{R}^{d})$ is dense in $\cX_{\rho}(\Omega)$.
\end{theorem}
We begin with the following lemma. 
\begin{lemma}\label{lma:Density1}
Assume that $\rho$ satisfies \eqref{eq:AssumptionsOnKernels1}. Suppose that ${\bf u} \in \cX_{\rho}(\Omega)$ with $\supp({\bf u}) \Subset \Omega$. Then there exists a sequence $ \{ {\bf u}_{n} \} \in C_{c}^{\infty}(\Omega; \mathbb{R}^{d})$ such that
\[
{\bf u}_{n} \to \bu \quad \text{in $\cX_{\rho}(\Omega)$ as $n\to \infty.$}
\]
\end{lemma}
\begin{proof}
Let $\varphi$ be the standard mollifier, and let $\varphi_\veps$ be the dilation of $\varphi$ by a factor of $\veps$.
Let $\eta =\dist(\supp({\bf u}), \Omega) > 0. $  Then for any $\veps  < \frac{\eta}{4},$ the vector field $\bu_\veps := \bu \ast \varphi_\veps \in C^{\infty}_c(\Omega;\bbR^d)$ with $\supp (\bu_\veps) \Subset \Omega$, and furthermore $[\bu_{\veps}]_{\cX_\rho(\Omega)} \leq [\bu]_{\cX_\rho(\Omega)}$ by an application of Jensen's inequality.
We claim that ${\bf u}_{\varepsilon} \to {\bf u}$ in $\cX_{\rho}(\Omega)$. Clearly ${\bf u}_{\veps} \to {\bf u}$ strongly in $L^{2}(\Omega;\mathbb{R}^{d})$, so we just need to check that $[{\bf u}_{\veps} - {\bf u}]_{\cX_{\rho}(\Omega)} \to 0$.

Fix $\tau > 0$. Since $\rho(\bx-\by) |\cD(\bu_\veps -\bu) (\bx,\by) \cdot \cD \bi(\bx,\by)|^2 \in L^1(\Omega \times \Omega)$, by continuity of the integral there exists a $\delta > 0$ such that
\begin{equation*}
    \int_\Omega \int_\Omega \chi_{ \{ |\by-\bx| < \delta \} } \rho(\bx-\by) |\cD(\bu_\veps -\bu) (\bx,\by) \cdot \cD \bi(\bx,\by)|^2 \, \rmd \by  \, \rmd \bx < \tau\,.
\end{equation*}
Now by \eqref{eq:AssumptionsOnKernels1} the function $\chi_{ \{ |\bz| < \delta \} } \rho(\bz) \in L^1_{loc}(\bbR^d)$, so therefore
\begin{equation*}
    \int_\Omega \int_\Omega \chi_{ \{ |\by-\bx| \geq \delta \} } \rho(\bx-\by) |\cD(\bu_\veps -\bu) (\bx,\by) \cdot \cD \bi(\bx,\by)|^2 \, \rmd \by  \, \rmd \bx \leq C(\delta) \Vnorm{\bu_\veps - \bu}_{L^2(\Omega)}^2\,.
\end{equation*}
We thus have that 
\begin{equation*}
    [\bu_\veps - \bu]_{\cX_\rho(\Omega)}^2 \leq C(\delta) \Vnorm{\bu_\veps - \bu}_{L^2(\Omega)}^2 + \tau\,.
\end{equation*}
Taking $\veps \to 0$ gives $\limsup_{\veps \to 0} [\bu_\veps - \bu]_{\cX_\rho(\Omega)}^2 \leq \tau$ for any $\tau > 0$, and the result is proved.
\end{proof}

\begin{proof}[Proof of Theorem \ref{global-boundary-approximation}] 
\noindent \underline{Step 1} (Local approximation):  Let $\bx_0 \in \p \Omega$. Then there exists $\gamma \in W^{1,\infty}(\bbR^{d-1})$, a radius $r > 0$, and an index $\sigma \in \{1,\ldots, d\}$ such that
    \begin{equation*}
        \Omega \cap B(\bx_0,r) = \{ \bx \in B(\bx_0,r) \, : \, x_\sigma > \gamma(x_1,\ldots,x_{\sigma-1}, x_{\sigma+1},\ldots, x_{d}) \}\,. 
    \end{equation*}
    Define $V := \Omega \cap B(\bx_0,\frac{r}{2})$. Now define the shifted point $\bx^{\veps} := \bx + \lambda \veps \be_\sigma$ for $\veps > 0$, where $\{ \be_i \}_{i=1}^d$ denotes the standard basis for $\bbR^d$. Then there exists $\veps_0 >0$ small and $\lambda > 0$ sufficiently large such that $B(\bx^{\veps},\veps) \subset \Omega \cap B(\bx_0,r)$ for all $\bx \in V$ and for all $\veps < \veps_0$.
    Define the shifted function $\bu_{\veps}(\bx) := \bu(\bx^{\veps})$ for $\bx \in V$. By a coordinate change, $[\bu_{\veps}]_{\cX_\rho(V)} \leq [\bu]_{\cX_\rho(\Omega)}$. Next define $\bv_{\veps} := \varphi_{\veps} \ast \bu_{\veps}$, where $\varphi_{\veps}$ is the dilation of the standard mollifier. Clearly $\bv_{\veps} \in C^{\infty}(\overline{V};\bbR^d)$, and by Jensen's inequality $[\bv_{\veps}]_{\cX_\rho(V)} \leq [\bu]_{\cX_\rho(\Omega)}$. Now we claim the following:
    \begin{enumerate}
        \item[i)] $\bv_{\veps} \to \bu \text{ in } \cX_\rho(V)$, and
        \item[ii)] for any $\zeta \in C^{\infty}_c(B(\bx_0,\frac{r}{2}))$ with $0 \leq \zeta \leq 1$, there exists $C = C(V,\zeta) > 0$ such that
        \begin{equation*}
            \Vnorm{\zeta \bv_{\veps} - \zeta \bu}_{\cX_\rho(\Omega)} \leq C \Vnorm{ \bv_{\veps} - \bu }_{\cX_\rho(V)}\,.
        \end{equation*}
    \end{enumerate}
    First we prove i). Write
    \begin{equation*}
        \Vnorm{\bv_{\veps} - \bu}_{\cX_\rho(V)} \leq \Vnorm{\bv_{\veps} - \bu_{\veps}}_{\cX_\rho(V)} + \Vnorm{\bu_{\veps} - \bu}_{\cX_\rho(V)}\,.
    \end{equation*}
    The first term converges to $0$ as $\veps \to 0$ by the same reasoning as in the proof of Lemma \ref{lma:Density1}. As for the second term, by continuity of translations in the $L^2$-norm we have $\Vnorm{\bu_{\veps} - \bu}_{L^2(V)} \to 0$ as $\veps \to 0$, so to prove i) we need to show that
    \begin{equation}\label{eq:Density2:Pf2}
        \limsup_{\veps \to 0} [\bu_{\veps} - \bu]_{\cX_\rho(V)} = 0\,.
    \end{equation}
    Let $\tau > 0$ be arbitrary. Then since $\rho(\bx-\by) |\cD(\bu_\veps -\bu) (\bx,\by) \cdot \cD \bi(\bx,\by)|^2 \in L^1(V \times V)$, by continuity of the integral there exists a $\delta > 0$ such that
    \begin{equation*}
    \int_V \int_V \chi_{ \{ |\by-\bx| < \delta \} } \rho(\bx-\by) |\cD(\bu_\veps -\bu) (\bx,\by) \cdot \cD \bi(\bx,\by)|^2 \, \rmd \by  \, \rmd \bx < \tau\,.
    \end{equation*}
    Now $\chi_{ \{ |\bz| < \delta \} } \rho(\bz) \in L^1_{loc}(\bbR^d)$ by \eqref{eq:AssumptionsOnKernels1}, so therefore
    \begin{equation*}
    \int_V \int_V \chi_{ \{ |\by-\bx| \geq \delta \} } \rho(\bx-\by) |\cD(\bu_\veps -\bu) (\bx,\by) \cdot \cD \bi(\bx,\by)|^2 \, \rmd \by  \, \rmd \bx \leq C(\delta) \Vnorm{\bu_\veps - \bu}_{L^2(\Omega)}^2\,.
    \end{equation*}
    We thus have that 
\begin{equation*}
    [\bu_\veps - \bu]_{\cX_\rho(V)}^2 \leq C(\delta) \Vnorm{\bu_\veps - \bu}_{L^2(\Omega)}^2 + \tau\,.
\end{equation*}
Taking $\veps \to 0$ gives $\limsup_{\veps \to 0} [\bu_\veps - \bu]_{\cX_\rho(V)}^2 \leq \tau$ for any $\tau > 0$, and \eqref{eq:Density2:Pf2} is proved.

Now we prove ii). Note that we use the convention $\zeta \bv_\veps(\bx) = \zeta \bu(\bx) = 0$ whenever $\bx \in \Omega \setminus V$. Since $\zeta \leq 1$ we have $\|\zeta (\bv_{\veps} - \bu) \|_{L^{2}(\Omega)} \leq  \|{\bf v}_{\veps} -{\bf u}\|_{L^{2}(V)}$.
Next we have
\[
\begin{split}
 [\zeta \bv_\veps - \zeta \bu]^{2}_{\cX_{\rho}(\Omega)}&= \int_{\Omega}\int_{\Omega} \rho(\bx-\by) |\cD(\zeta\bv_\veps -\zeta \bu) (\bx,\by) \cdot \cD \bi(\bx,\by)|^2 \, \rmd \by \, \rmd \bx\\
    &= \int_{V}\int_{V} \rho(\bx-\by) |\cD(\zeta \bv_\veps - \zeta \bu) (\bx,\by) \cdot \cD \bi(\bx,\by)|^2 \, \rmd \by \, \rmd \bx\\
    &+ 2 \int_{\Omega \setminus V} \int_{V} \rho(\bx-\by) \zeta^2(\by) | (\bv_\veps(\by) - \bu(\by) ) \cdot \cD \bi(\bx,\by)|^2 \, \rmd \by \, \rmd \bx \\
    & := J_{1}^{\veps} +J_{2}^{\veps}.
\end{split}
\]
Writing $ \cD(\zeta \bv_\veps - \zeta \bu) (\bx,\by) = \zeta(\bx) \cD(\bv_\veps -\bu) (\bx,\by) + (\bv_\veps(\by)-\bu(\by)) \cD\zeta(\bx,\by) $, we have
\[
J_{1}^{\veps} \leq \|\zeta\|^{2}_{L^{\infty}(V)}[\bv_\veps - \bu]_{\cX_{\rho}(V)}^2 + C(V) \|\nabla \zeta\|^{2}_{L^{\infty}(V)}\|{\bf v}_\veps - {\bf u}\|^{2}_{L^{2}(V)}\,.
\]
Moreover, since $R := \dist( \supp(\zeta), (\Omega\setminus V)) > 0$,
\[
\begin{split}
J_{2}^{\veps} &\leq 2 \int_{\Omega\setminus V}\int_{V\cap \supp(\zeta)} \rho(\bx-\by) \zeta^2(\by) | \bv_\veps(\by) - \bu(\by)|^2 \, \rmd \by \, \rmd \bx\\
&\leq 2\|\zeta\|^{2}_{L^{\infty}} \left(\int_{\{|\bsxi| > R\}}\rho(\bsxi) \, \rmd \bsxi \right) \left(\int_{V} |\bv_\veps(\by) - \bu(\by)|^2 \, \rmd \by\right)\,,
\end{split}
\]
concluding the proof of ii).

\noindent\underline{Step 2} (Global approximation): 
Let $\delta > 0$. Since $\p \Omega$ is compact, we can choose finitely many points $\bx_i \in \p \Omega$, radii $r_i = r(\bx_i) > 0$, and corresponding sets $V_i = \Omega \cap B(\bx_i,\frac{r_i}{2})$ so that $\p \Omega \subset \cup_{i=1}^N B(\bx_i,\frac{r_i}{2})$. By Step 1, we can find corresponding functions $\bv_i \in C^{\infty}(\overline{V}_i;\bbR^d)$ such that for $i = 1,\ldots,N$  
\begin{equation*}
    \Vnorm{\bv_i - \bu}_{\cX_\rho(V_i)} \leq \delta\,. 
\end{equation*}
Let $V_0 \Subset \Omega$ be an open set such that $\Omega \subset \cup_{i=0}^N V_i$ and choose a function $\bv_0 \in C^{\infty}(\overline{V}_0;\bbR^d)$ according to Lemma \ref{lma:Density1} that satisfies
\begin{equation*}
    \Vnorm{\bv_0 - \bu}_{\cX_\rho(V_0)} \leq \delta\,. 
\end{equation*}
Let $\{ \zeta_i \}_{i=0}^N$ be a $C^\infty$ partition of unity subordinate to the sets $\{ V_0, B(\bx_1,\frac{r_1}{2}), \ldots, B(\bx_N,\frac{r_N}{2}) \}$, and define $\bv := \sum_{i=0}^N \zeta_i \bv_i$.
Then $\bv \in C^{\infty}(\overline{\Omega};\bbR^d)$, and since $\bu = \sum_{i=0}^N \zeta_i  \bu $ we have
using Step 1, ii) as well as Lemma \ref{lma:Density1}
\begin{equation*}
    \Vnorm{ \bv - \bu }_{\cX_\rho(\Omega)} \leq \sum_{i = 0}^N C(V_i,\zeta_i) \Vnorm{\bv_i - \bu}_{\cX_\rho(V_i)} \leq C(N+1) \delta\,.
\end{equation*}
That concludes the proof of the theorem.
\end{proof}

\begin{proof}[Proof of Theorem \ref{thm:Linearization}]
Recall the functionals $E_\veps$ and $E_{0}$ defined in \eqref{nl-total-energy} and \eqref{quad-energy}, respectively, and their corresponding extensions $\overline{E}_\veps$ and $\overline{E}_0$. The result follows from the inequalities
\[ \overline{E}_{0} 
\leq \GammaliminfL2_{\veps \to 0} \overline{E}_{\veps}
\leq \GammalimsupL2_{\veps \to 0} \overline{E}_{\veps}\leq \overline{E}_{0}, 
\]
where $\GammaliminfL2_{\veps \to 0} \overline{E}_{\veps}$ and $\GammalimsupL2_{\veps \to 0} \overline{E}_{\veps}$ are the lower and upper $\Gamma$-limits, respectively, of $ \overline{E}_{\veps}$.
The second inequality is trivial, and the first and third inequalities are established in Lemma \ref{lma:Linearization:Liminf} and Lemma \ref{lma:Linearization:Limsup} respectively.
\end{proof}

\begin{lemma}\label{lma:Linearization:Liminf}
    Let $\bu \in L^{2}(\Omega;\bbR^d)$.   For every sequence $\bu_{\veps}$ converging to $\bu$ strongly in $L^2(\bbR^d)$, we have
    \begin{equation*}
        \overline{E}_{0}(\bu) \leq \liminf_{\veps \to 0} \overline{E}_{\veps}(\bu_{\veps})\,.
    \end{equation*}
\end{lemma}
\begin{proof}
There exists a measure zero set $N\subset \Omega\times\Omega$ such that for all $(\bx, \by)\in (\Omega\times\Omega)\setminus N$
\[
\cD\bu_{\veps}(\by, \bx)=\frac{\bu_\veps(\by)-\bu_\veps(\bx)}{|\by-\bx|} \to \frac{\bu(\by)-\bu(\bx)}{|\by-\bx|}=\cD\bu(\by, \bx) 
\]
as $\veps\to 0.$ This means that for a fixed $(\bx,\by)\in (\Omega\times\Omega)\setminus N$, the sequence of vectors $\cD\bu_\veps(\by, \bx)$ remains a bounded set in $\bbR^d.$
After noting that 
\[s_m[\bi + \veps \bu_{\veps}](\by,\bx) = \wt{s}_m(\cD\bi(\by, \bx),\veps\,\cD\bu_\veps(\by,\bx)),\]
 by Corollary \ref{pointwise-limit-density} (or directly \eqref{loc-limit-fixed-xieta}),  we have that for a fixed $(\bx,\by)\in (\Omega\times \Omega)\setminus N$,  
 \[
\lim_{\veps\to0}\frac{1}{\veps^2} w(\by-\bx, s_m[\bi + \veps \bu_{\veps}](\by,\bx)) = \rho(\by-\bx) \left(\cD\bu(\by, \bx) \cdot \cD \bi(\by,\bx)\right)^2. 
 \]
Since $w$ is nonnegative, integrating and applying Fatou's lemma we have 
\[
\liminf_{\veps\to 0}{1\over \veps^2}\int_{\Omega}\int_{\Omega}w(\by-\bx, s_m[\bi + \veps \bu_{\veps}](\by,\bx))\, \rmd \by \, \rmd \bx \geq \int_{\Omega}\int_{\Omega} \rho(\by-\bx) \left(\cD\bu(\by, \bx) \cdot \cD \bi(\by,\bx)\right)^2\, \rmd \by \, \rmd \bx. 
\]
Notice that the right hand side is finite only when $\bu\in \cX_\rho(\Omega)$ by definition \eqref{dot-space}. When $\bu\in L^{2}(\Omega;\bbR^d)\setminus \cX_\rho(\Omega)$, the right hand side becomes $\infty$ and so is the left hand side of the inequality.  
Finally, we use the $L^2$-continuity of the linear map $\bu \mapsto \int_{\Omega}\bl(\bx)\cdot \bu(\bx)\, \rmd \bx$ to conclude 
\[\liminf_{\veps\to 0} \overline{E}_{\veps}(\bu_\veps) \geq \overline{E}_0(\bu). \]
\end{proof}
\begin{lemma}\label{lma:Linearization:Limsup}
    For any $\bu\in L^2(\Omega;\bbR^d)$, 
    \begin{equation*}
        \GammalimsupL2_{\veps \to 0} \overline{E}_\veps(\bu) \leq \overline{E}_0(\bu)\,.
    \end{equation*}
\end{lemma}

\begin{proof}
The inequality is trivial if $\bu\in L^{2}(\Omega;\bbR^d)\setminus \cX_\rho(\Omega)$, so we assume that $\bu \in \cX_\rho(\Omega)$. We first assume additionally that $\bu \in W^{1,\infty}(\Omega;\bbR^d)$. Set $M := \Vnorm{\bu}_{W^{1,\infty}(\Omega)}$. Then $|\cD[\bu](\bx,\by)| \leq M$ for almost every $(\bx,\by) \in \Omega \times \Omega$.
    Then, by Corollary \ref{pointwise-limit-density}, $\overline{E}_{\veps}(\bu) < \infty$ 
    for $\veps > 0$ small enough and 
    by the dominated convergence theorem
    \begin{equation*}
        \lim\limits_{\veps \to 0} E_\veps(\bu) = E_0(\bu)\,.
    \end{equation*}
    For the general case $\bu \in \cX_\rho(\Omega)$, by Theorem \ref{global-boundary-approximation}, there exists a sequence $\{ \bu_k \} \subset W^{1,\infty}(\Omega;\bbR^d)$ such that $\bu_k \to \bu$ strongly in $\cX_\rho (\Omega)$, and so by the lower semicontinuity of the $\Gamma$-$\limsup$ and the continuity of $E_0$ with respect to strong convergence in $\cX_\rho(\Omega;\bbR^d)$
    \begin{equation*}
        \GammalimsupL2_{\veps \to 0} E_\veps(\bu) \leq \liminf_{k \to \infty} \left( \GammalimsupL2_{\veps \to 0} E_\veps(\bu_k) \right) \leq \liminf_{k \to \infty} E_0(\bu_k) = E_0(\bu)\,.
    \end{equation*}
\end{proof}

\section{The localization limit}\label{sec:Loc}

In this section, we prove the first part of Theorem \ref{thm:Intro:GammaConvergence}. That is, we prove that the $\Gamma$-limit of the sequence of functionals $\overline{F}_n$ has an integral representation.  Throughout the section we assume that $\rho_n$ is a sequence satisfying \eqref{eq:AssumptionsOnKernels1} and that $\Phi$ 
satisfies \eqref{eq:PDEnergy:ConditionsOnG}.

Let $\Omega \subset \bbR^d$ be a bounded domain and define $\cA_0(\Omega) := \{\Omega' \subset \Omega \, : \, \Omega' \text{ is open}\}$. Given $A \in \cA_0(\Omega)$, we introduce the localized energy
\begin{equation*}
	F_n(\bv,A) := \int_A \int_{A} \rho_n(\bx-\by) \Phi(|s[\bv](\bx,\by)|) \, \rmd \by \, \rmd \bx\,. 
\end{equation*}
Notice that $F_{n}(\bv, \Omega) = F_{n}(\bv)$, as defined in \eqref{p-growing-nl-energy}.  
Extend the energies $F_n(\cdot,A)$ to all of $L^{p}(\Omega;\bbR^d)$ by defining 
$\overline{F}_n : L^{p}(\Omega; \bbR^d) \times \cA_0(\Omega) \to \overline{\bbR}$ as
\begin{equation*}
	\overline{F}_n(\bv,A) = 
	\begin{cases}
		F_n(\bv,A)\,, & \text{ if } \bv \in \mathfrak{W}^{\rho_n,p}(\Omega;\bbR^d); \\
		+\infty\,, & \text{ otherwise. }
	\end{cases}
\end{equation*}
Note that this is a natural definition of the extended energies on $L^{p}(\Omega;\bbR^d)$, since by definition of $s[\bv](\by,\bx)$
and by \eqref{eq:PDEnergy:ConditionsOnG}
\begin{equation}\label{eq:CoercivityOfEnergy}
	\begin{split}
		[\bv]_{\mathfrak{W}^{\rho_n,p}(A)}^{p} \leq C(p) ( F_n(\bv,A) + |A| ).
	\end{split}
\end{equation}
By the compactness property of $\Gamma$-convergence {\cite[Chapter 7]{braides1998homogenization}}, there exists a subsequence, not relabeled, and a lower semicontinuous functional  $\overline{F}_{\infty} : L^{p}(\Omega;\bbR^d) \to \overline{\bbR}$ such that $\overline{F}_n$ $\Gamma$-converges in the strong $L^{p}(\Omega;\bbR^{d})$ topology to $\overline{F}_\infty$. 
We do not know how to explicitly compute $\overline{F}_\infty$, but we can prove that this limiting functional has an integral representation. Inspired by \cite[Theorem 3.1 and 3.3]{alicandro2004general}, we prove the following preliminary ``local version'' of Theorem \ref{thm:Intro:GammaConvergence}, which gives the $\Gamma$-limit representation for any set $A \in \cA_0(\Omega)$. 
\begin{theorem}[Local $\Gamma$-convergence]\label{thm:LocalGammaConvergence}
	There exists a quasiconvex Carath\'eodory function $f_\infty : \bbR^{d \times d} \to \bbR$ satisfying the growth condition
	\begin{equation*}
		0 \leq f_\infty(\bbF) \leq c_2 (|\bbF|^{p} + 1)
	\end{equation*}
	such that with respect to the strong topology on $L^{p}(\Omega;\bbR^d)$
	\begin{equation*}
		\GammalimLp_{n \to \infty} \overline{F}_n(\bv,A) =  \overline{F}_\infty(\bv,A)
	\end{equation*}
    holds for any $A \in \cA_0(\Omega)$, where $\overline{F}_\infty(\cdot,A) : L^{p}(\Omega;\bbR^d) \to \overline{\bbR}$ is defined as
	\begin{equation*}
		\overline{F}_\infty(\bv,A) := \begin{cases}
			\intdm{A}{f_\infty(\grad \bv(\bx))}{\bx}\,, & \bv \in W^{1,p}(\Omega;\bbR^d)\\
			+\infty\,, & \text{ otherwise. }
		\end{cases}
	\end{equation*}
\end{theorem}
The proof of the theorem follows a general framework wherein we establish sufficient conditions that allow for the application of the De Georgi-Letta measure criterion, and thus conclude that $\overline{F}_\infty(\bv,\cdot)$ is a restriction of a Borel measure on $\cA_0(\Omega)$ for any fixed $\bv\in W^{1,p}(\Omega;\bbR^d)$. The general integral representation theorem that contains the required sufficient conditions to be proven is the following:
\begin{theorem}{\cite[Theorem 2.2]{alicandro2004general}}\label{thm:ACCriterion}
	Let $1 \leq p < \infty$, and let $\cF : W^{1,p}(\Omega;\bbR^d) \times \cA_0(\Omega) \to [0,\infty]$ be a functional that satisfies the following:
	\begin{enumerate}[i)]
		\item $\cF(\bv,A) = \cF(\bw,A)$ if $\bv = \bw$ on $A \in \cA_0(\Omega)$.
		\item For every $\bv \in W^{1,p}(\Omega;\bbR^d)$ the function $\cF(\bv,\cdot)$ is the restriction of a Borel measure to $\cA_0(\Omega)$.
		\item There exist $C > 0$ and $a \in L^1(\Omega)$ such that
		\begin{equation*}
			\cF(\bv,A) \leq C \intdm{A}{(a(\bx)+|\grad \bv|^p)}{\bx}
		\end{equation*}
		for every $\bv \in W^{1,p}(\Omega;\bbR^d)$ and every $A \in \cA_0(\Omega)$.
		\item $\cF(\bv,A) = \cF(\bv+\bz,A)$ for every $\bz \in \bbR^d$, $\bv \in W^{1,p}(\Omega;\bbR^d)$ and $A \in \cA_0(\Omega)$
		\item $\cF(\cdot,A)$ is lower semicontinuous with respect to weak convergence in $W^{1,p}(\Omega;\bbR^d)$, for every $A\in \cA_0(\Omega)$.
	\end{enumerate}
	Then there exists a Carath\'eodory function $f_0 : \Omega \times \bbR^{d \times d} \to [0,\infty)$ satisfying the growth condition
	\begin{equation*}
		0 \leq f_0(\bx,\bbF) \leq C(a(\bx) + |\bbF|^p)
	\end{equation*}
	for all $\bx \in \Omega$ and $\bbF \in \bbR^{d \times d}$ such that $f(x, \cdot)$ is quasiconvex and 
	\begin{equation*}
		\cF(\bv,A) = \intdm{A}{f_0(\bx,\grad \bv(\bx))}{\bx}
	\end{equation*}
	for all $\bv \in W^{1,p}(\Omega;\bbR^d)$ and $A \in \cA(\Omega)$.
	If in addition 
	\begin{enumerate}[i), resume]
		\item $\cF(\bbF \bx,B(\bx_0,r)) = \cF(\bbF \bx,B(\by_0,r))$
		for every $\bbF \in \bbR^{d \times d}$ and $\bx_0$, $\by_0 \in \Omega$ and $r > 0$ such that $B(\bx_0,r) \cup B(\by_0,r) \subset \Omega$
	\end{enumerate}
	then $f_0$ does not depend on $\bx$.
\end{theorem}

To prepare for the proof of Theorem \ref{thm:LocalGammaConvergence}, we define the lower and upper $\Gamma$ limits of the functionals. Thoughout this section, we work on the subsequence (not relabeled) for which $\overline{F}_n$ $\Gamma$-converges to $\overline{F}_{\infty}$ for $A = \Omega$. By selecting a diagonal subsequence, we may even choose the subsequence in such a way that the convergence holds over any subset of $\Omega$ that belongs to a countable dense subcollection $\mathfrak{D}(\Omega)$ of $\cA_0(\Omega)$. \footnote{By a dense subcollection $\mathfrak{D}(\Omega)$ of $\cA_0(\Omega)$ we mean for any $A, B\in \cA_0(\Omega)$, such that $A\Subset B$ there exists $D_0\in \mathfrak{D}(\Omega)$ so that $A\Subset D\Subset B$. }
Now fix  
$A \in \cA_0(\Omega)$ and define the lower $\Gamma$-limit $F'$ and upper $\Gamma$-limit $F''$ as
\begin{equation*}
	\begin{split}
		F'(\bv,A) := \GammaliminfLp_{n \to \infty} \overline{F}_n(\bv,A) &= \inf \left\{ \liminf_{n \to \infty} \overline{F}_n(\bv_n,A) \, : \, \bv_n \to \bv \text{ in } L^{p}(\Omega;\bbR^d) \right\}\,, \\
		F''(\bv,A) := \GammalimsupLp_{n \to \infty} \overline{F}_n(\bv,A) &= \inf \left\{ \limsup_{n \to \infty} \overline{F}_n(\bv_n,A) \, : \, \bv_n \to \bv \text{ in } L^{p}(\Omega;\bbR^d) \right\}\,.
	\end{split}
\end{equation*}
Then $\overline{F}_\infty(\bv) = F'(\bv,\Omega)= F''(\bv,\Omega)$. Moreover, for any $A\in \mathfrak{D}_0(\Omega)$, we have
\begin{equation*}
 \overline{F}_\infty(\bv,A) = F'(\bv,A)=F''(\bv,A). 
\end{equation*}
The next lemma gives a lower bound for $F'$ for any $A\in \cA_0(\Omega)$ and $\bv\in W^{1,p}(\Omega;\bbR^{d})$  such that $F'(\bv, A) < \infty$. In particular, it implies  that $\overline{F}_\infty(\bv) = + \infty$ whenever $\bv \notin W^{1,p}(\Omega;\bbR^d)$. 
\begin{lemma}\label{lma:BoundFromBelow}
	There exists a positive constant $C(d,p)$ such that for any $\bv \in L^{p}(\Omega;\bbR^d)$ and $A\in \cA_0(\Omega)$ with $F'(\bv,A) < \infty$ 
	\begin{equation*}
		F'(\bv,A) \geq C ( \Vnorm{\grad \bv}_{L^{p}(A;\bbR^d)}^p - |A| )\,.
	\end{equation*}
\end{lemma}

\begin{proof}
	Let $\bv_n$ be a sequence converging to $\bv$ strongly in $L^p(\Omega;\bbR^d)$. Then  using the $p$-growth condition on $\Phi,$ we have 
	\begin{equation*}
		\liminf_{n \to \infty} [\bv_n]_{\frak{W}^{\rho_n,p}(A)}^{p} \leq C  \liminf_{n \to \infty} F_n(\bv_n,A) + C|A| \,.
	\end{equation*}
	Since the sequence $\bv_n$ was arbitrary,
	\begin{equation*}
		\liminf_{n \to \infty} [\bv_n]_{\frak{W}^{\rho_n,p}(A)}^p \leq C  F'(\bv,A) + C |A| \,.
	\end{equation*}
 Since $F'(\bv,A) < \infty$ by assumption, we have that (up to subsequences) 
 \[
\sup_{n \in \bbN}[\bv_n]_{\frak{W}^{\rho_n,p}(A)}^p < \infty.
 \]
This implies that the sequence of vector fields $\{\bv_n\}$ is compact in $L^{p}(A;\bbR^d)$ with a limit vector field $\bv \in W^{1,p}(A;\bbR^d)$ (see \cite{BBM}) with the estimate
		\begin{equation*}
			\intdm{\Omega}{ \fint_{\bbS^{d-1}}  |\grad \bv(\bx) \bsomega|^{p}  \, \rmd \sigma(\bsomega)}{\bx} \leq C  I'(\bv,A) + C|A| \,.
	\end{equation*}
The conclusion of the lemma follows from the equivalence of the norms $ \left( \fint_{\bbS^{d-1}}  |\bbF \bsomega|^p  \, \rmd \sigma(\bsomega) \right)^{1/p}$ and $|\bbF|$ on $\bbR^{d \times d}$.
\end{proof}
Next, we will obtain a lower bound for $F''$.
\begin{lemma}\label{lma:BoundFromAbove}
	There exists a positive constant $C(d,p)$ such that for any $\bv \in W^{1,p}(A;\bbR^d)$ and $A\in \cA_0(\Omega)$ 
	\begin{equation*}
		F''(\bv,A) \leq C ( \Vnorm{\grad \bv}_{L^{p}(A)}^{p} + |A| )\,.
	\end{equation*}
\end{lemma}
\begin{proof}
	By the $p$-growth condition \eqref{eq:PDEnergy:ConditionsOnG} on $\Phi$ and H\"older's inequality
	\begin{equation*}
		\begin{split}
		F_n(\bv,A) 
		&\leq  C_1 \int_A \int_{\Omega} \rho_n(\bx-\by) ( 1 + |s[\bv](\bx,\by)|^{p} ) \, \rmd \by \, \rmd \bx \\
		&\leq C \iintdm{A}{\Omega}{\rho_n(\bx-\by) \left| \cD[\bv](\bx,\by) \right|^{p} }{\by}{\bx} + C\iintdm{A}{\Omega}{\rho_n(|\bx-\by|)}{\by}{\bx} \\
		&\leq C \int_A \int_{\Omega} \rho_n(\by-\bx) \frac{|\bv(\by)-\bv(\bx)|^{p}}{|\by-\bx|^{p}} \, \rmd \by \, \rmd \bx + C |A|\,,
		\end{split}
	\end{equation*}
	where the constant $C$ is independent of $n$ and $\bv$.
	Using the same methods as in \cite{BBM}
	\begin{equation*}
		\lim_{n \to \infty} \int_A \int_{\Omega} \rho_n(\by-\bx) \frac{|\bv(\by)-\bv(\bx)|^{p}}{|\by-\bx|^{p}} \, \rmd \by \, \rmd \bx
		= \int_A \fint_{\bbS^{d-1}}  \left| \grad \bv(\bx) \cdot \bsomega \right|^{p} \, \rmd \sigma(\bsomega) \, \rmd \bx \leq \Vnorm{\grad \bv}_{L^{p}(A)}^{p}\,.
	\end{equation*}
	Thus
	\begin{equation*}
		\limsup_{n \to \infty} F_n(\bv,A) \leq C ( \Vnorm{\grad \bv}_{L^{p}(A)}^{p} + |A| )\,,
	\end{equation*}
	and the result follows from the definition of $F''$.
\end{proof}
We next prove the subadditivity of the upper $\Gamma$-limit. 
\begin{lemma}\label{lma:Subadditivity}
	Let $A$, $A'$, $B$, $B' \in \cA_0(\Omega)$ with $A' \Subset A$ and $B' \Subset B$. Then for any $\bv \in W^{1,{p}}(\Omega;\bbR^d)$
	\begin{equation*}
		F''(\bv,A' \cup B') \leq F''(\bv,A) + F''(\bv,B)\,.
	\end{equation*}
\end{lemma}

\begin{proof}
Without loss of generality, we assume that $F''(\bv,A)$ and $ F''(\bv,B)$ are finite. 
	Let $q = \dist(A',\p A)$. Fix $M \in \bbN$, and for $k \in \{ 1, \ldots, M\}$ define
	\begin{equation*}
		A_k := \left\{ \bx \in A \, : \, \dist(\bx,A') < \frac{kq}{M} \right\}\,, \qquad A_0 = A'\,.
	\end{equation*}
	So $A_M \subset A$. For each $k$ define the cutoff function $\varphi_k \in C^{\infty}_c(\Omega)$ such that $\varphi_k \equiv 1$ in $A_{k-1}$, $0 \leq \varphi_k \leq 1$ on $A_k \setminus A_{k-1}$ and $\varphi_k \equiv 0$ on $\Omega \setminus A_k$, so that $\Vnorm{\grad \varphi_k}_{L^{\infty}(\Omega)} \leq \frac{2M}{q}$. Let $\bv_n$ and $\bw_n$ be two sequences that converge to $\bv$ strongly in $L^{p}(\Omega;\bbR^d)$ with the property that
	\begin{equation*}
		\limsup_{n \to \infty} F_n(\bv_n,A) = F''(\bv,A)\,, \qquad 	\limsup_{n \to \infty} F_n(\bw_n,B) = F''(\bv,B)\,.
	\end{equation*}
	Note that by \eqref{eq:CoercivityOfEnergy} the sequences of seminorms $[\bv_n]_{\frak{W}^{\rho_n,p}(A)}$ and $[\bw_n]_{\frak{W}^{\rho_n,p}(B)}$ are uniformly bounded. Next,  define
	\begin{equation*}
		\bszeta_{k,n} := \varphi_k \bv_n + (1- \varphi_k) \bw_n\,.
	\end{equation*}
	For fixed $k$, $\bszeta_{k,n}$ converges to $\bv$ strongly in $L^{p}(\Omega;\bbR^d)$ as $n \to \infty$. Also, 
    by definition, we have
\begin{equation}\label{eq:Lemma5.6PfSplitUp}
		\cD \bszeta_{k,n}(\bx,\by) = \varphi_k(\by) \cD \bv_n (\bx,\by) + (1- \varphi_k(\by) ) \cD \bw_n (\bx,\by)  + (\bv_n(\bx) -\bw_n(\bx)) \cD \varphi_k (\bx,\by)\,.
	\end{equation}
It then follows by Jensen's inequality for finite sums and the $L^{\infty}$ bound on $\varphi_k$ and $\grad \varphi_k$ that 
	\begin{equation*}
		|\cD \bszeta_{k,n}(\bx,\by)|^{p} \leq 2^{p-1} (|\cD \bv_n(\bx,\by)|^{p} + |\cD \bw_n(\bx,\by)|^{p} + \frac{2M}{q} |\bv_n(\by) - \bw_n(\by)|^{p} )
	\end{equation*}
	for all $(\bx,\by) \in \Omega \times \Omega$.
Notice that, from \eqref{eq:Lemma5.6PfSplitUp}, 
	\begin{equation*}
		\begin{split}
			\cD \bszeta_{k,n}(\bx,\by) =
				\begin{cases}
					\cD\bv_n(\bx,\by)\,, & (\bx,\by) \in A_{k-1} \times A_{k-1}\,, \\
					\cD \bw_n(\bx,\by)\,, & (\bx,\by) \in (\Omega \setminus A_k ) \times (\Omega \setminus A_k)\,. \\
				\end{cases}
		\end{split}
	\end{equation*}
	Now, for all $k = 0 ,\ldots, M$, we have from $B' = (A_{k-1} \cap B') \cup ( (\Omega\setminus \overline{A}_k) \cap B' ) \cup ( (A_k \setminus \overline{A}_{k-1}) \cap B'  )$ that we can write $A'\cup B'$ as a disjoint union as 
	\begin{equation*}
		\begin{split}
			A' \cup B' 
			&= (A' \cup (A_{k-1} \cap B')) \cup ( (\Omega\setminus \overline{A}_k) \cap B' ) \cup ( (A_k \setminus \overline{A}_{k-1}) \cap B'  )
		\end{split}
	\end{equation*}
	Setting $T_1 = A' \cup (A_{k-1} \cap B')$, $T_2 = (\Omega\setminus \overline{A}_k) \cap B'$ and $T_3 = (A_k \setminus \overline{A}_{k-1}) \cap B'$, we have
	\begin{equation*}
		\begin{split}
			&F_n(\bszeta_{k,n}, A' \cup B') \\
			&= F_n( \bv_{n}, T_1 ) + F_n( \bw_n, T_2 ) \\
				& \quad + \iint_{ ((A' \cup B') \times (A' \cup B')) \setminus ((T_1 \times T_1) \cup (T_2 \times T_2)) } \rho_n(\bx-\by) \Phi( |s[\bszeta_{k,n}](\bx,\by)| ) \, \rmd \by \, \rmd \bx \\
			&\leq F_n(\bv_n,A) + F_n(\bw_n,B) \\
				&\quad + C\iint_{ ((A' \cup B') \times (A' \cup B')) \setminus ((T_1 \times T_1) \cup (T_2 \times T_2)) } \rho_n(\bx-\by) (1 + |\cD \bszeta_{k,n}(\bx,\by)|^{p}) \, \rmd \by \, \rmd \bx \\
			&\leq F_n(\bv_n,A) + F_n(\bw_n,B) \\
				&\quad + C\iint_{ ((A' \cup B') \times (A' \cup B')) \setminus ((T_1 \times T_1) \cup (T_2 \times T_2)) } \rho_n(\bx-\by) (1 + |\cD \bv_n(\bx,\by)|^{p} + |\cD \bw_n(\bx,\by)|^{p}) \, \rmd \by \, \rmd \bx \\
				&\quad + C\iint_{ ((A' \cup B') \times (A' \cup B')) \setminus ((T_1 \times T_1) \cup (T_2 \times T_2)) } \rho_n(\bx-\by) \frac{2M}{q} |\bv_n(\by) - \bw_n(\by)|^{p} \, \rmd \by \, \rmd \bx\,.
		\end{split}
	\end{equation*}
	The third integral is majorized by
	\begin{equation}\label{eq:Lemma5.6:MainEstimate1}
		\frac{C(m,p)M}{q} \Vnorm{\bv_n - \bw_n}_{L^{p}(\Omega)}^{p}\,.
	\end{equation}
	To estimate the second integral we decompose the set $ ((A' \cup B') \times (A' \cup B')) \setminus ((T_1 \times T_1) \cup (T_2 \times T_2)) = S_{1,k} \cup S_{2,k}$, where
	\begin{equation*}
		\begin{split}
		S_{1,k} &:= (T_1 \times T_2) \cup (T_2 \times T_1) \\
		&\qquad \cup (A' \times T_3) \cup (T_3 \times A') \\
		&\qquad 
		\cup [( (\Omega \setminus \overline{A}_{k+1})  \cap B' ) \times T_3] \cup [T_3 \times ( (\Omega \setminus \overline{A}_{k+1})  \cap B' )]  \\
		&\qquad \cup [(A_{k-2} \cap B')  \times T_3 ] \cup [T_3 \times (A_{k-2} \cap B')]
		\end{split}
	\end{equation*}
	and
	\begin{equation*}
		S_{2,k} = [((A_{k-1} \setminus \overline{A}_{k-2}) \cup (A_{k+1} \setminus \overline{A}_{k}) ) \cap B' \times T_3] \cup [T_3 \times ( ( A_{k+1} \setminus \overline{A}_{k-2} ) \cap B') ]
	\end{equation*}
	are disjoint sets.
	Observe first that $\dist( S_{1,k}, \diag ) > 0$ for $k \in \{1, \ldots, M-2\}$, where $\diag := \{ (\bx,\bx)  \, : \, \bx \in \Omega\}$. The distance between the sets depends on $M$, and at worst is comparable to $\frac{1}{M-2}$. So let $q_0(M)$ be a number that satisfies $q_0(M) \leq \dist( S_{1,k}, \diag )$ for $k \in \{1, \ldots, M-2\}$.
	We therefore have that	\begin{equation}\label{eq:Lemma5.6:MainEstimate2}
		\begin{split}
		\iint_{S_{1,k}} &\rho_n(\bx-\by) (1 + |\cD \bv_n(\bx,\by)|^{p} + |\cD \bw_n(\bx,\by)|^{p}) \, \rmd \by \, \rmd \bx \\
		&\leq C \left( |\Omega| + \frac{\Vnorm{\bv_n}_{L^{p}(\Omega)}^{p}}{q_0(M)^{p}} + \frac{\Vnorm{\bw_n}_{L^{p}(\Omega)}^{p}}{q_0(M)^{p}} \right) \int_{|\bsxi| \geq q_0(M)} \rho_n(\bsxi) \, \rmd \bsxi\,.
		\end{split}
	\end{equation}
	Now observe that any point in $\bigcup_{k=3}^{M-4} S_{2,k}$ is contained in no more than $8$ of the $S_{2,k}$. 	Since $A_{M-3} \cap B' \subset (A \cap B)$, we have that 
	\begin{equation*}
		\begin{split}
			\sum_{k=3}^{M-4} \iint_{S_{2,k}} &\rho_n(\bx-\by) (1 + |\cD \bv_n(\bx,\by)|^{p} + |\cD \bw_n(\bx,\by)|^{p}) \, \rmd \by \, \rmd \bx \\
			&\leq 8 \int_{ A_{M-3} \cap B' } \int_{ A_{M-3} \cap B' } \rho_n(\bx-\by) (1 + |\cD \bv_n(\bx,\by)|^{p} + |\cD \bw_n(\bx,\by)|^{p}) \, \rmd \by \, \rmd \bx \\
			&\leq 8 (|\Omega| + [\bv_n]_{\frak{W}^{\rho_n,p}(A)}^p + [\bw_n]_{\frak{W}^{\rho_n,p}(B)}^p) \leq C_0
		\end{split}
	\end{equation*}
	for some constant $C_0$ independent of $n$.
	Thus we can choose an index $k = k(n) \in \{ 3, \ldots, M-4 \}$ such that 
	\begin{equation}\label{eq:Lemma5.6:MainEstimate3}
		\iint_{S_{2,k(n)}} \rho_n(\bx-\by) (1 + |\cD \bv_n(\bx,\by)|^{p} + |\cD \bw_n(\bx,\by)|^{p}) \, \rmd \by \, \rmd \bx \leq \frac{C_0}{M-4}\,.
	\end{equation}
	Combining \eqref{eq:Lemma5.6:MainEstimate1}, \eqref{eq:Lemma5.6:MainEstimate2} and \eqref{eq:Lemma5.6:MainEstimate3} brings us to
	\begin{equation*}
		\begin{split}
			F_n( \bszeta_{k(n),n},A' \cup B') &\leq F_n(\bv_n,A) + F_n(\bw_n,B) \\
			&\qquad + \frac{CM}{q} \Vnorm{\bv_n - \bw_n}_{L^{p}(\Omega)}^{p} + \frac{C_0}{M-4} \\
			&\qquad + 
			C \left( |\Omega| + \frac{\Vnorm{\bv_n}_{L^{p}(\Omega)}^{p}}{q_0(M)^{p}} + \frac{\Vnorm{\bw_n}_{L^{p}(\Omega)}^{p}}{q_0(M)^{p}} \right) \int_{|\bsxi| \geq q_0(M)} \rho_n(\bsxi) \, \rmd \bsxi\,. 
		\end{split}
	\end{equation*}
	Note that $\bszeta_{k(n),n} \to \bv$ strongly in $L^{p}(\Omega;\bbR^d)$, and that $\lim\limits_{n \to \infty} \intdm{|\bsxi|> q_0(M) }{\rho_n(\bsxi)}{\bsxi} = 0$. Therefore taking the limit as $n \to \infty$ gives
	\begin{equation*}
		F''(\bv,A' \cup B') \leq F''(\bv,A) + F''(\bv,B) + \frac{C_0}{M-4}\,,
	\end{equation*}
	and the result follows by letting $M \to \infty$.
\end{proof}
For our first result, first notice that both the lower and upper $\Gamma$-limits $F'(\bv,\cdot)$ and $F''(\bv, \cdot)$ are increasing set functions. The next lemma shows that, in fact, $F''$ is inner regular. 
\begin{lemma}\label{lma:InnerMeasure}
	For $\bv \in W^{1,p}(\Omega;\bbR^d)$ and for $A \in \cA_0(\Omega)$,
	$
		\sup_{A' \Subset A} F''(\bv,A') = F''(\bv,A)\,.
	$
\end{lemma}

\begin{proof}
	 It suffices to show that $
		\sup_{A' \Subset A} F''(\bv,A') \geq F''(\bv,A)\,.$  
	Let $\delta > 0$. Then there exists a set $A'' \Subset A$ such that
	\begin{equation}\label{eq:Lemma5.7:Pf1}
		|A\setminus \overline{A''}| + \Vnorm{\grad \bv}_{L^{p}( A \setminus \overline{A''} )}^{p} < \delta\,.
	\end{equation}
	Let $A' \in \cA_0(\Omega)$ with $A'' \Subset A' \Subset A$. Let $\bv_n$ and $\bw_n$ be two sequences in $L^{p}(\Omega;\bbR^d)$ that converge to $\bv$ strongly in $L^{p}(\Omega;\bbR^d)$ with the property that
	\begin{equation*}
		\lim_{n \to \infty} F_n(\bv_n,A') = F''(\bv,A') \quad \text{ and } \quad \lim_{n \to \infty} F_n(\bw_n, A \setminus \overline{A''}) = F''(\bv, A \setminus \overline{A''})\,.
	\end{equation*}
	Note that by Lemma \ref{lma:BoundFromAbove} and \eqref{eq:Lemma5.7:Pf1} there exists a quantity $R_n$ with $\lim\limits_{n \to \infty} R_n \to 0$ such that
	\begin{equation}\label{eq:Lemma5.7:Remainder}
		F_n(\bw_n, A \setminus \overline{A''}) \leq F''(\bv, A \setminus \overline{A''}) + R_n \leq C \delta + R_n\,.
	\end{equation}
Arguing in exactly the same manner as in the proof of Lemma \ref{lma:Subadditivity} (by taking $B=A \setminus \overline{A''}$ and $A=A'$), we can construct a sequence $\bszeta_{k(n),n}$ (as a combination of $\bv_n$ and $\bw_n$) that converges to $\bv$ strongly in $L^{p}(\Omega;\bbR^d)$ as $n \to \infty$ and 
\begin{equation*}
		\begin{split}
			F_n( \bszeta_{k(n),n},A) &\leq F_n(\bv_n,A') + F_n(\bw_n, A \setminus \overline{A''}) \\
			&\qquad + \frac{CM}{q} \Vnorm{\bv_n - \bw_n}_{L^{p}(\Omega)}^{p} + \frac{C_0}{M-4} \\
			&\qquad + 
			C \left( |\Omega| + \frac{\Vnorm{\bv_n}_{L^{p}(\Omega)}^{p}}{q_0(M)^{p}} + \frac{\Vnorm{\bw_n}_{L^{p}(\Omega)}^{p}}{q_0(M)^{p}} \right) \int_{|\bsxi| \geq q_0(M)} \rho_n(\bsxi) \, \rmd \bsxi\,. 
		\end{split}
	\end{equation*}
 Therefore using \eqref{eq:Lemma5.7:Remainder} and taking the limit as $n \to \infty$  we obtain 
	\begin{equation*}
		F''(\bv,A) \leq F''(\bv,A') + C\delta + \frac{C_0}{M-4}\,.
	\end{equation*}
Finally,  taking $M \to \infty$ gives
	\begin{equation*}
		F''(\bv,A) \leq F''(\bv,A') + C\delta \leq \sup_{A' \Subset A} I''(\bv,A') + C \delta\,,
	\end{equation*}
	and the result follows since $\delta > 0$ is arbitrary.

\end{proof}
\begin{remark}\label{identification-subaditivity}
As a corollary of the previous two lemmas we obtain the following. First, Lemma \ref{lma:InnerMeasure} allows us to extend $\overline{F}_{\infty}(\bu, \cdot)$ as an inner regular measure as follows. Fix $\bu\in W^{1,p}(\Omega;\bbR^d)$. Define the inner regular envelope $ {F}_{\infty}^{env}(\bu, \cdot): \mathcal{A}_0(\Omega) \to \mathbb{R}$ of $\overline{F}_{\infty} (\bu, \cdot)$ as follows
    \[
    {F}_{\infty}^{env}(\bu, A) = \sup_{A'\Subset A, A\in \frak{D}_0(\Omega)} \overline{F}_{\infty}(\bu, A')=\sup_{A'\Subset A, A\in \frak{D}_0(\Omega)} F'(\bu, A')=\sup_{A'\Subset A, A\in \frak{D}_0(\Omega)} F''(\bu, A').
    \]
    By density of the class $\frak{D}_0(\Omega)$ in $\cA_{0}(\Omega)$, we have that 
    \[
     {F}_{\infty}^{env}(\bu, A) = \sup_{A'\Subset A} F'(\bu, A')=\sup_{A'\Subset A} F''(\bu, A').
    \]
  Moreover,  since $\overline{F}_{\infty}(\bu, A') = F''(\bu, A')$ for all $A'\in \frak{D}_0(\Omega)$ we have 
    \[
       {F}_{\infty}^{env}(\bu, A) = \sup_{A'\Subset A, A\in \frak{D}_0(\Omega)} F''(\bu, A') = \sup_{A'\Subset A} F''(\bu, A') = F''(\bu, A)
    \]
    where in the last equality we used Lemma \ref{lma:InnerMeasure}. In addition, using the monotonocity of $F'(\bu, \cdot)$ as a set function we have 
    \[
    {F}_{\infty}^{env}(\bu, A) = \sup_{A'\Subset A} F'(\bu, A')\leq F'(\bu, A) \leq F''(\bu, A)=  {F}_{\infty}^{env}(\bu, A)\]
   As a consequence, from hereafter we assume that $\overline{F}_{\infty}$ is inner regular and identify $\overline{F}_{\infty}$, $F^{env}_{\infty}$, and $F''$ as the same at all $\bu$ and all $A\in \cA_0(\Omega).$ Additionally, from this identification and Lemma \ref{lma:Subadditivity} it is now clear that for any $\bu\in W^{1,p}(\Omega;\bbR^d)$ and $A, B\in \cA_0(\Omega)$
    \[
    \overline{F}_{\infty}(\bu, A\cup B) \leq \overline{F}_{\infty}(\bu, A) + \overline{F}_{\infty}(\bu, B).
    \]
    In the event the sets $A$ and $B$ are disjoint, we have equality. This is proved in the next lemma. 
\end{remark}

\begin{lemma}\label{superadditivity}
    Let $\bu \in W^{1,p}(\Omega;\bbR^{d})$. Suppose that $A, B\in \cA_0(\Omega)$ with $A\cap B = \emptyset$. Then 
    \[
    \overline{F}_{\infty}(\bu, A\cup B) \geq \overline{F}_{\infty}(\bu, A) + \overline{F}_{\infty}(\bu, B).
    \]
    
\end{lemma}
\begin{proof}
Let $A', B'\in \cA_0(\Omega)$ with $A'\Subset A$, $B'\Subset B$. Set $\delta = \dist(A', B')>0$.  Choose a sequence $\bu_{n} \to \bu$ strongly in $L^p(\Omega;\bbR^d)$ such that 
\[
F'(\bu, A'\cup B') = \liminf_{n \to \infty} F_{n}(\bu_n, A'\cup B'). 
\]
Then for each $n$
    \begin{equation}\label{super-add-equation}
    F_{n}(\bu_n, A'\cup B') = F_{n}(\bu_n, A')+F_{n}(\bu_n, B') + 2\int_{A'}\int_{B'} \rho_{n}(\bx-\by)\Phi(|s[\bu_n](\bx, \by)|) \, \rmd \by \, \rmd \bx\,.
    \end{equation}
    Notice that using the fact that $|\bx-\by|>\delta$ for $\bx\in A'$ and $\by\in B'$, 
    \[
    \begin{split}
    2\int_{A'}\int_{B'} \rho_{n}(\bx-\by)\Phi(|s[\bu_n](\bx, \by)|) \, \rmd \by \, \rmd \bx &\leq C \int_{A'}\int_{B'} \rho_{n}(\bx-\by)(1 + |\cD  \bu_n (\bx,\by)|^p) \, \rmd \by \, \rmd \bx\\
    &\leq C (1 + \delta^{-p}\|\bu_n\|^p_{L^{p}}) \int_{\{|\bsxi|>\delta\}}\rho_{n}(|\bsxi|)\,. 
    \end{split}
    \]
    The latter goes to $0$ as $n\to \infty$ since $\rho_n$ satisfies \eqref{eq:AssumptionsOnKernels1}. Now taking the limit inferior as $n \to \infty$ on both sides of the equation \eqref{super-add-equation} we obtain by definition of the lower $\Gamma$-limit $F'$ that
    \[
    F'(\bu, A') + F'(\bu, B') \leq F'(\bu, A'\cup B'). 
    \]
    We now take the sup over $A' \Subset A$ and $B' \Subset B$ and use the inner regularity of $\overline{F}_{\infty}$ to complete the proof.
\end{proof}
The next lemma shows that the $\Gamma$-limit $\overline{F}_{\infty}$ is in fact local. To be precise, we have the following. 
\begin{lemma}\label{locality}
    Let $A\in \cA_0(\Omega)$. Suppose that $\bu, \bv \in W^{1,p}(\Omega;\bbR^d)$ such that $\bu=\bv$ for almost all $\bx\in A$. Then $\overline{F}_{\infty}(\bu, A) = \overline{F}_{\infty}(\bv, A)$.
\end{lemma}

\begin{proof}
    It suffices to show that $\overline{F}_{\infty}(\bu, A) \leq \overline{F}_{\infty}(\bv, A)$ (as we can switch the role of $\bu$ and $\bv$), for which, in turn, it is enough to prove that
    \[
    F''(\bu, A') \leq F''(\bv, A'), \; \text{for any $A'\Subset A$}.
    \] 
    To that end, fix $A', A''$ such that $A'\Subset A''\Subset A$. Set $\delta=\dist(A', \bbR^{d}\setminus A'')$. Choose $\bv_n \to \bv $ strongly in $L^{p}(\Omega;\bbR^d)$ such that 
    \[
    F''(\bv, A') = \limsup_{n\to \infty} F_{n}(\bv_n, A'). 
    \]
    Let $\varphi\in C_{c}^{\infty}(A)$ such that $\varphi = 1$ on $A''$. Construct the sequence $\bu_n = \varphi \bv_n + (1-\varphi) \bu$. Since $\bu =\bv$ almost everywhere in $A$, we have that $\bu_n\to \bu$ strongly in $L^{p}(\Omega;\bbR^d)$. Using this sequence, it is clear now that 
    \[
    F''(\bu, A') \leq \limsup_{n\to \infty} F_{n}(\bu_n, A')= \limsup_{n\to \infty} F_{n}(\bv_n, A') = F''(\bv, A'),
    \]
    where we used $\bu_n = \bv_n$ on $A'$. 
\end{proof}
We are now ready to give the proof the integral representation for the $\Gamma$-limit $\overline{F}_{\infty}$. 
\begin{proof}[Proof of Theorem \ref{thm:LocalGammaConvergence}]
	First, it follows from Lemma \ref{lma:BoundFromBelow} that $\overline{F}_{\infty}(\bv,A) = +\infty$ if $\bv \in L^{p}(A;\bbR^d) \setminus W^{1,p}(A;\bbR^d)$.
	Now we will check each of the sufficient conditions i) through vi) of Theorem \ref{thm:ACCriterion}. To that end, item $i)$ is precisely Lemma \ref{locality}. For item $ii)$, we apply the De Georgi-Letta Measure Criterion. For that, we have shown previously that the $\Gamma$-limit  $\overline{F}_{\infty}(\bv,\cdot)$ is an increasing set function,  inner regular,  subadditive and superadditive, see Remark \ref{identification-subaditivity} and Lemma \ref{superadditivity}. Item $iii)$ follows from the same upper bound for the $\Gamma$-$\limsup$ (Lemma \ref{lma:BoundFromAbove}). Item $iv)$ follows from the fact that $F_n$ depend only on difference quotients of $\bv$. Finally, item $v)$ follows from the fact that $\overline{F}_{\infty}$ is a $\Gamma$-limit and the fact that weak convergence in $W^{1,p}(\Omega;\bbR^d)$ implies strong convergence in $L^{p}(\Omega;\bbR^d)$ by Rellich's theorem. Thus Theorem \ref{thm:ACCriterion} guarantees the existence of a   Carath\'eodory function $f_\infty : \Omega \times \bbR^{d \times d} \to [0,\infty)$ satisfying the growth condition
	\begin{equation*}
		0 \leq f_\infty(\bx,\bbF) \leq C(a(\bx) + |\bbF|^p)
	\end{equation*}
	for all $\bx \in \Omega$ and $\bbF \in \bbR^{d \times d}$ such that $f_\infty(x, \cdot)$ is quasiconvex and 
	\begin{equation*}
		\overline{F}_{\infty}(\bv,A) = \intdm{A}{f_\infty(\bx,\grad \bv(\bx))}{\bx}
	\end{equation*}
	for all $\bv \in W^{1,p}(\Omega;\bbR^d)$ and $A \in \cA_0(\Omega)$. 
 What remains to show is item vi), and so conclude that the integrand $f_\infty$ does not depend on $\bx$.
To that end, let $\bx_0$, $\by_0 \in \Omega$ and $r > 0$ such that $B(\bx_0,r) \cup B(\by_0,r) \subset \Omega$ be given. We need to show that 
$\overline{F}_{\infty}(\ba,B(\bx_0,r)) = \overline{F}_{\infty}(\ba,B(\by_0,r))$
for every $\bbF \in \bbR^{d \times d}$, where the vector field $\ba$ is the affine map given by $\ba(\bx)=\bbF \bx$.
To prove this, we first show that
\begin{equation}\label{eq:TranslationInvariantProof0}
		\overline{F}_{\infty}(\ba, B(\bx_{0}, r')) \leq \overline{F}_{\infty}(\ba, B(\by_{0}, r) )
	\end{equation}
for any $0 < r' < r$. 
	Let ${\bf v}_{n}$ be a sequence of vector fields converging to the affine map $\ba$ strongly in $L^{p}(\Omega;\mathbb{R}^{d})$ such that 
	\[
	\liminf_{n \to \infty} F_{n}({\bf v}_{n}, B(\by_{0}, r)) = \overline{F}_{\infty}(\ba, B(\by_{0}, r))\,.
	\]
	Define a new sequence
	\begin{equation*}
		{\bf u}_{n}(\bx) := 
		\begin{cases}
			{\bf v}_{n}(\bx+\by_{0}-\bx_{0}) - \bbF (\by_{0} - {\bx_{0}})\,, &\quad \text{ when } \bx \in B(\bx_{0}, r)\,, \\
			\bbF\bx\,, &\quad \text{ otherwise. }
		\end{cases}
	\end{equation*}
	Observe that $\bx\in B(\bx_{0}, r)$ if and only if $\bx+\by_{0} - \bx_{0}\in B(\by_{0}, r)$.
	Then a straightforward calculation shows that ${\bf u}_{n} \to \ba$ strongly in $L^{p}(\Omega,\mathbb{R}^{d})$.  Moreover, for each $n$ we have by change of variables
	\[
	\begin{split}
		F_{n}({\bf u}_{n}, B(\bx_{0}, r')) &=  \int_{ B(\bx_{0}, r')}  \int_{B(\bx_{0}, r')} \rho_{n}(\by-\bx) \Phi( |s[\bu_n](\by,\bx)| )  \, \rmd \by \, \rmd \bx \\
		&= \int_{B(\bx_{0}, r')} \int_{B(\bx_0,r')}   \rho_{n}(\by-\bx) \Phi( |s[\bv_n](\by+\by_0-\bx_0,\bx+\by_0-\bx_0)| ) \, \rmd \by \, \rmd \bx \\
		&=\int_{B(\by_{0}, r')} \int_{B(\by_0,r')}   \rho_{n}(\by-\bx) \Phi (|s[\bv_n](\by,\bx)|)\, \rmd \by \, \rmd \bx\,.
	\end{split}
	\]
	So we we obtain by monotonicity of the double integral that 
	\begin{equation}\label{eq:TranslationInvariantProof1}
		F_{n}({\bf u}_{n} , B(\bx_{0}, r')) \leq F_{n}({\bf v}_{n}, B(\by_{0}, r))\,.
	\end{equation}
	Taking the liminf on both sides of \eqref{eq:TranslationInvariantProof1}, 
	\begin{equation*}
		\overline{F}_{\infty}(\ba, B(\bx_0,r')) \leq \liminf_{n \to \infty} F_n( \bv_n , B(\bx_0,r')) \leq \liminf_{n \to \infty} F_n( \bv_n , B(\by_0,r)) = \overline{F}_{\infty}(\ba, B(\by_0,r))\,,
	\end{equation*}
	which is \eqref{eq:TranslationInvariantProof0}.
	Now,  $r' < r$ was arbitrary, so an application of Lemma \ref{lma:InnerMeasure} results in
	\begin{equation*}
		\overline{F}_{\infty}(\ba, B(\bx_0,r)) \leq \overline{F}_{\infty}(\ba, B(\by_0,r))\,.
	\end{equation*}
	We obtain equality in the above inequality by noting that $\bx_{0}$ and $\by_{0}$ were arbitrary. 
\end{proof}

We close this section by commenting on the convergence of the sequence of minimizers to the nonlocal functions $F_n$ subject to Dirichlet boundary conditions.  
To that end, for a given function $\bg \in W^{1,p}(\Omega;\bbR^d)$, and for a given $A \in \cA_0(\Omega)$ and a fixed $r_0 \in (0,\frac{1}{2}\diam(A))$, define
\begin{equation*}
	\frak{W}^{\rho_n,p}_{\bg}(A;\bbR^d) := \{ \bv \in \frak{W}^{\rho_n,p}(\Omega;\bbR^d) \, ; \, \bv(\bx) = \bg(\bx) \text{ for } \bx \in \Omega \text{ with } \dist(\bx,\Omega \setminus A) < r_0 \}\,,
\end{equation*}
and define $F_{n,\bg} : L^{p}(\Omega;\bbR^d) \times \cA_0(\Omega) \to \overline{\bbR}$ by
\begin{equation*}
	F_{n,\bg}(\bv,A) := 
		\begin{cases}
			F_n(\bv,A)\,, & \bv \in \frak{W}^{\rho_n,p}_{\bg}(\Omega;\bbR^d) \,, \\
			+\infty\,, & \text{otherwise. }
		\end{cases}
\end{equation*}
We use the direct method of calculus of variations to show that under an additional condition on the sequence of kernels, for each $n$, the functional $F_{n,\bg}(\cdot, A)$ has a minimizer in $\mathfrak{W}^{\rho_n,p}_{\bg}(A)$. 
\begin{proposition}
    For a fixed $n \in \bbN$,
    suppose that $\rho_n$ satisfies
    \begin{equation}\label{eq:Assumption:DensityCondition}
        \lim\limits_{\delta \to 0} \int_{ \{ |\bz| > \delta \} } \, \frac{\rho_n(\bz)}{|\bz|^p} \rmd \bz  = \infty
   \end{equation}
    in addition to satisfying \eqref{eq:AssumptionsOnKernels1}. Then 
    there exists $\bv \in \frak{W}^{\rho_n,p}_{\bg}(A;\bbR^d)$ such that
    \begin{equation}\label{eq:ExistenceOfMin}
        F_{n,\bg}(\bv,A) = \min_{ \bw \in \frak{W}^{\rho_n,p}_{\bg}(A;\bbR^d) } F_{n,\bg}(\bw,A)\,.
    \end{equation}
\end{proposition}

\begin{proof}
    The proof uses direct methods. Fix $n \in \bbN$, and let $\{ \bv_k \}$ be a minimizing sequence of the functional $\bv\mapsto F_{n, \bg}(\bv, A).$  Then
    \begin{equation*}
        [\bv_k]_{\frak{W}^{\rho_n,p}(A)}^p \leq C \left( F_{n,\bg}(\bv_k,A) + |A| \right)
    \end{equation*}
    for all $k \in \bbN$, and $\bv_k = \bg$ on $\Omega \setminus A$. Therefore $[\bv_k]_{\frak{W}^{\rho_n,p}(\Omega)}^p$ is a bounded sequence in $k$.
    Furthermore, by the nonlocal Poincar\'e inequality \cite{Ponce-poincare,BBM} 
    \begin{equation*}
        \Vnorm{\bv_k}_{L^p(\Omega)} \leq \Vnorm{ \bv_k - \bg }_{L^p(\Omega} + \Vnorm{\bg}_{L^p(\Omega)} \leq C \left( [\bv_k]_{\frak{W}^{\rho_n,p}(\Omega)} + \Vnorm{\bg}_{L^p(\Omega)} \right)\,,
    \end{equation*}
    and so $\{ \bv_k \}_k $ is bounded in $\frak{W}^{\rho_n,p}(\Omega;\bbR^d)$.
    Therefore
    by an application of the compactness lemma in \cite[Theorem A.1]{Du-Mengesha-Tian-compactness}  for kernels $\rho_{n}$ satisfying \eqref{eq:AssumptionsOnKernels1} and \eqref{eq:Assumption:DensityCondition}, $\bv_k$ converges strongly in $L^p_{loc}(\Omega;\bbR^d)$ to a function $\bv \in L^{p}_{loc}(\Omega;\bbR^d)$. But $\bv_k = \bg$ for all $k\geq 1$ and all $\bx \in \Omega$ with $\dist(\bx, \Omega \setminus A) < r_0$, the convergence is global in $L^{p}(\Omega;\mathbb{R}^d)$ and so $\bv\in \frak{W}^{\rho_n,p}_{\bg}(\Omega;\bbR^{d})$. Moreover, by Fatou's lemma applied to the sequence $\{ s[\bv_k](\bx,\by) \}_k$
    \begin{equation*}
        F_{n,\bg}(\bv,A) \leq \liminf_{k \to \infty} F_{n,\bg}(\bv_k,A)\,.
    \end{equation*}
    It follows that $\bv$ satisfies \eqref{eq:ExistenceOfMin}
\end{proof}
Now following a procedure similar to the one described in this section, one can prove that 
	\begin{equation*}
		\GammalimLp_{n \to \infty} F_{n,\bg}(\bv,A) = F_{\infty,\bg}(\bv,A)\,,
	\end{equation*}
	where
	\begin{equation*}
		F_{\infty,\bg}(\bv,A) := \begin{cases}
			\intdm{A}{f_\infty(\grad \bv(\bx))}{\bx}\,, & \bv - \bg \in W^{1,p}_0(A;\bbR^d)\\
			+\infty\,, & \text{otherwise, }
		\end{cases}
	\end{equation*}
	where the $\Gamma$-limit is taken with respect to the strong topology on $L^{p}(\Omega;\bbR^d)$. Moreover, given a sequence of $\bv_n\in \frak{W}^{\rho_n,p}_{\bg}(A;\bbR^d)$ such that $\{F_{n,\bg}(\bv_n,A)\}_n$ is uniformly bounded, we have that $\{\bv_n\}_n$ is compact in the strong topology of $L^{p}(\Omega;\mathbb{R}^{d})$ \cite{Ponce-poincare, BBM} with a limit $\bv$ such that $\bv-\bg \in W^{1,p}_{0}(A;\bbR^{d})$. Thus the sequence of functionals $\{F_n(\cdot, A)\}_n$ is equicoercive in the strong topology on $L^p(\Omega;\bbR^d)$. We may now apply \cite[Theorem 7.8 and Corollary 7.202]{DalMaso} to conclude that minimizers of $F_{n, \bg}(\cdot, A)$ converge to a minimizer of $F_{\infty, \bg}(\cdot, A)$. We summarize these results in the following proposition. 

\begin{proposition}
	With the assumptions of this section, 
	suppose a sequence $\{\bv_n \} \subset L^{p}(\Omega;\bbR^d)$ satisfies
	\begin{equation*}
		F_{n,\bg}(\bv_n,A) = \min \{ F_{n,\bg}(\bw,A) \, : \, \bw \in \frak{W}^{\rho_n,p}_{\bg}(A;\bbR^d) \}\,.
	\end{equation*}
	Then there exists a function $\bv\in L^{p}(\Omega;\mathbb{R}^{d})$ satisfying $\bv - \bg \in W^{1,p}_0(A;\bbR^d)$ such that $\bv_n \to \bv$ strongly in $L^{mp}(\Omega;\bbR^d)$ and that
	\begin{equation*}
		F_{\infty,\bg}(\bv,A) = \min \{ F_{\infty,\bg}(\bw,A) \, : \, \bw-\bg \in W^{1,p}_{0}(A;\bbR^d) \}\,.
	\end{equation*}
\end{proposition}

\begin{remark}\label{gamma-conv-remark}
    Using the generalized nonlinear strain $s_{m}[\bv](\bx, \by)$ for $m\geq 1$ introduced in Section \ref{sec:Intro}, one can also study the variational convergence of functionals of the form
    \[F_{n}^{m}(\bv):=\int_\Omega \int_{\Omega} \rho_n(\bx-\by) \Phi(|s_m[\bv](\bx,\by)|) \, \rmd \by \, \rmd \bx\,, \quad \bv\in \frak{W}^{\rho_n,mp}(\Omega;\bbR^d)\,.
 \]
 Following the same procedure as above, one can show that the $\Gamma$-limit in the strong topology of $L^{mp}(\Omega;\mathbb{R}^d)$ of 
 the sequence of the extended functionals $\overline{F}^{m}_{n}$ has an integral representation given by 
 \[\overline{F}_{\infty}^{m}(\bv)= \int_{\Omega} f_{\infty}^{m}(\nabla \bv (\bx)) \, \rmd \bx \,, \quad \text{for all $\bv\in W^{1,mp}(\Omega;\bbR^{d})$}, 
 \]
 where $f^{m}_{\infty}: \bbR^{d\times d}\to \mathbb{R}$ is a quasiconvex functional with $mp$-growth. When $m=1$, we denote $\overline{F}_{\infty}^{1}$ and $f^{1}_{\infty}$ by $\overline{F}_{\infty}$ and $f_{\infty}$ respectively, coinciding with the notation previously set. 
\end{remark}

\section{Properties of the limiting functional and its density}\label{sec:LocProp}

The proof of Theorem \ref{thm:Intro:GammaConvergence} will follow from Theorem \ref{thm:LocalGammaConvergence} with $A = \Omega$ along with further properties of the $\Gamma$-limit proved in Theorem \ref{thm:InvarianceOfIntegrand}, Theorem \ref{thm:ZeroSetOfFxnal}, and Corollary \ref{cor:CoercivityOfIntegrand} below.

\subsection{Estimate for the limiting functional}
In the previous section, we have shown that the $\Gamma$-limit of the sequence of functionals $\overline{F}_{n}$ has an integral representation with a density function $f_{\infty}$ that is quasiconvex. Computing an explicit characterization of $f_{\infty}$ in terms of the density function $\Phi$ and the sequence $\rho_n$ is a nontrivial task. 
However, the functionals $F_n$ lie in the general framework of nonlocal models studied in \cite{ponce2004new}. In that work, a successful attempt has been made to compute the density function for the $\Gamma$-limit (with respect to the strong topology on $L^1(\bbR^d)$ and $C^1(\bbR^d)$) in the very special cases when $\bv\mapsto \Phi(s[\bv])$ is convex or concave. This result is not applicable in our case as  $\bv\mapsto \Phi(s[\bv])$  is nonconvex.
In the same work \cite{ponce2004new} and for general $\Phi$, upper and lower bounds for the $\Gamma$-limit are obtained. 
The next proposition states these results; we sketch the proof following the same techniques used in \cite{ponce2004new} but adapted for the vector-valued case in the strong $L^p(\Omega;\bbR^d)$ topology for the specific form of $\Phi$.
\begin{proposition} \label{thm:BoundsOnPeridynamicEnergy}
	Suppose that $1<p < \infty$ and let $\overline{F}_{\infty}$ be as in Theorem \ref{thm:LocalGammaConvergence}.
 Then 
	\begin{equation*}
		F_{\text{LB}}(\bv) \leq \overline{F}_{\infty}(\bv) \leq F_{\text{UB}}(\bv) \qquad \text{ for any } \bv \in W^{1,p}(\Omega;\bbR^d),
	\end{equation*}
	where the lower and upper bound functionals $F_{LB}$ and $F_{UB}$ are defined as 
	\begin{equation*}
		F_{\text{LB}}(\bv) := \int_{\Omega}  \fint_{\bbS^{d-1}} \Phi((|\nabla \bv (\bx) \bsomega|-1)_+) \, \rmd \sigma(\bsomega) \, \rmd \bx,\quad\quad (x_+ = \max \{ x,0 \})
	\end{equation*}
    and
	\begin{equation*}
		F_{\text{UB}}(\bv) := \int_{\Omega} Q \wt{\Phi}(\grad \bv(\bx)) \, \rmd \bx
	\end{equation*}
	where $Q \wt{\Phi}$ is the quasiconvexfication of $\wt{\Phi}$ and  
\begin{equation}\label{eq:BoundsOnIntegrandDefns2}
  \wt{\Phi}(\bbF) :=  \fint_{\bbS^{d-1}} {\Phi}(||\bbF \bsomega| -1|) \, \rmd \sigma(\bsomega) \,.
	\end{equation}
\end{proposition}

\begin{proof}
We show the upper bound first. To begin, by the compact embedding of $W^{1,p}(\Omega;\bbR^d)$ in $L^{p}(\Omega;\mathbb{R}^d)$ we have
\[
\GammalimsupLp_{n\to \infty} F_{n}(\bv) \leq \GammalimsupW1p_{n\to \infty} F_{n}(\bv),\quad \text{for $\bv\in W^{1,p}(\Omega;\bbR^d)$,}
\]
where the latter is the upper $\Gamma$-limit with respect to the weak topology in $W^{1,p}(\Omega;\bbR^d)$.  Next, we will show that 
\begin{equation}\label{lim-eqn-bound}
\GammalimsupW1p_{n\to \infty} F_{n}(\bv)\leq \int_{\Omega} \wt{\Phi}(\nabla \bv (\bx)) \, \rmd \bx 
\quad\text{for any $\bv\in W^{1,p}(\Omega;\bbR^d)$}
\end{equation}
where $\wt{\Phi}$ is as given in \eqref{eq:BoundsOnIntegrandDefns2}. 
By the weak lower semicontinuity of $\GammalimsupW1p_{n\to \infty} F_{n}$, it follows that 
\[
\GammalimsupW1p_{n\to \infty} F_{n}(\bv)\leq\int_{\Omega} Q\wt{\Phi} (\nabla \bv (\bx)) \, \rmd \bx\,,
\]
which proves the upper bound. To establish \eqref{lim-eqn-bound}, we take $\bv_n = \bv\in  W^{1,p}(\Omega;\bbR^d)$, and we claim that it serves as a recovery sequence, i.e.
\begin{equation}\label{limit-compute-formula}
\lim_{n\to \infty} \int_{\Omega}\int_{\Omega}\rho_{n}(\by-\bx)\Phi(|s[\bv] (\bx, \by)|)\, \rmd \by \, \rmd \bx = \int_{\Omega}  \fint_{\bbS^{d-1}} {\Phi}(||\nabla \bv(\bx) \bsomega| -1|) \, \rmd \sigma(\bsomega) \, \rmd \bx.
\end{equation}
To that end, first note that using the convexity and $p$-growth of $\Phi$, for any $\bv, \bw \in W^{1,p}(\Omega;\bbR^d)$ and $\bx,\by \in \Omega$, we have that (see \cite[Proposition 4.64]{fonseca2007modern})
\[
\Phi(|s[\bv](\bx, \by)|)-\Phi(|s[\bw](\bx, \by)|) \leq C(1 + |\cD\bv(\bx,\by)|^{p-1}  + |\cD\bw (\bx, \by)|^{p-1})|\cD(\bv - \bw) (\bx, \by)|.
\]
Applying H\"older's inequality, we have that for each $n \in \bbN$
\[
\begin{split}
\lim_{n\to \infty} &\left|\int_{\Omega}\int_{\Omega}\rho_{n}(\by-\bx)\Phi(|s[\bv] (\bx, \by)|)\, \rmd \by \, \rmd \bx -  \int_{\Omega}\int_{\Omega}\rho_{n}(\by-\bx)\Phi(|s[\bw] (\bx, \by)|)\, \rmd \by \, \rmd \bx\right|\\
&\leq C(1 + \|\nabla \bv\|^{p-1}_{L^{p}(\Omega)} + \|\nabla \bw\|^{p-1}_{L^{p}(\Omega)})\|\nabla \bv-\nabla \bw\|_{L^{p}(\Omega)}.
\end{split}
\]
Therefore, by density of $C^{2}(\overline{\Omega};\mathbb{R}^{d})$ in $W^{1,p}(\Omega;\bbR^d)$, it suffices to demonstrate \eqref{limit-compute-formula} for $\bv\in C^{2}(\overline{\Omega};\mathbb{R}^{d})$. But this has been done in \cite[Proposition 4.1]{ponce2004new}, see also \cite[Theorem 2]{BBM}. The upper bound is therefore established. 

Next, we prove the lower bound by proving that 
\begin{equation}\label{limit-compute-lower}
\GammaliminfLp_{n\to \infty}F_{n}(\bv) \geq \int_{\Omega} \fint_{\bbS^{d-1}} C\varphi(|\nabla \bv (\bx) \bsomega|) \, \rmd \sigma(\bsomega) \, \rmd \bx
\end{equation}
where $C\varphi$ is the convexification of $\varphi(t) = \Phi(||t|-1|).$ To that end, let $\bv\in W^{1,p}(\Omega;\bbR^d)$ and $\bv_{n} \to \bv$ in $L^{p}(\Omega;\bbR^{d})$. Then we have that 
\[
F_{n}(\bv_n) = \int_{\Omega}\int_{\Omega}\rho_n (\by-\bx)\varphi (|\cD\bv_n (\bx, \by)|)\, \rmd \by \, \rmd \bx \geq  \int_{\Omega}\int_{\Omega} \rho_n (\by-\bx) C\varphi (|\cD\bv_n (\bx, \by)|) \, \rmd \by \, \rmd \bx.
\]
Now let $\eta\in C_{c}(B(\boldsymbol{0},1))$ be a radial function such that $\int_{B(\boldsymbol{0},1)} \eta(\bz) \,  \rmd \bz = 1$ and $\eta_{\delta} (\bz) = {1\over \delta^d} \eta\left({|\bz|\over \delta}\right)$ for $\delta > 0$ serve as an approximation to the identity. Then by Jensen's inequality we have that for any $r>0$ small, for all $\delta <r$, for any $n \geq 1$
\[
\int_{\Omega}\int_{\Omega} \rho_n (\by-\bx) C\varphi (|\cD\bv_n (\bx, \by)|) \, \rmd \by \, \rmd \bx \geq \int_{\Omega_r}\int_{\Omega_r} \rho_n (\by-\bx) C\varphi (|\cD(\eta_\delta \ast \bv_n) (\bx, \by)|) \, \rmd \by \, \rmd \bx
\]
where $\Omega_r = \{\bx\in \Omega: \dist(\bx, \partial \Omega) >r \}$. For a fixed $\delta < r$, 
$\eta_\delta \ast \bv_n \to \eta_\delta \ast \bv$ in $C^{2}(\overline{\Omega_r}),$ and so for any $\delta <r$,
\[
\lim_{n\to \infty} \int_{\Omega_r}\int_{\Omega_r} \rho_n (\by-\bx) C\varphi (|\cD(\eta_\delta \ast \bv_n) (\bx, \by)|) \, \rmd \by \, \rmd \bx = \int_{\Omega_r} \fint_{\bbS^{d-1}} C\varphi (|\nabla (\eta_\delta \ast \bv) (\bx) \bsomega|) \, \rmd \sigma(\bsomega) \, \rmd \bx.
\]
We now let $\delta\to 0$ to conclude that
\[
\liminf_{n\to \infty} F_n(\bv_n) \geq \int_{\Omega_r} \fint_{\bbS^{d-1}} C\varphi (|\nabla\bv (\bx)\bsomega| ) \, \rmd \sigma(\bsomega)  \, \rmd \bx. 
\]
Inequality \eqref{limit-compute-lower} now follows after letting $r\to 0$. The formula for the convexification of $\varphi$ in terms of the function $\Phi$ follows from Carath\'eodory's theorem that characterizes the convexification 
    \begin{equation*}
        C\varphi(t) = \inf \left\{ \lambda \varphi(t_1) + (1-\lambda) \varphi(t_2) \, : \, t = \lambda t_1 + (1-\lambda)t_2\,, 0 \leq \lambda \leq 1 \right\}\,.
    \end{equation*}
    First, suppose $t$ is in the range $(-\infty,-1) \cap (1,\infty)$. Then the function $\big| |t| - 1 \big|$ is convex on that range.
    Therefore the function $\varphi(t) := \Phi(\big| |t| - 1 \big|)$, being the composition of the nondecreasing convex function $\Phi$ with $\big| |t| - 1 \big|$, is convex on that range. Hence $C\varphi(t) = \varphi(t)$ for $t \in (-\infty,-1) \cap (1,\infty)$. For $t \in [-1,1]$, since $0 \leq \varphi(t)$ for all $t$, it suffices to show that for all $t \in [-1,1]$ we can find $\lambda \in [0,1]$, $t_1$ and $t_2$ with $t = \lambda t_1 + (1-\lambda) t_2$ such that $\lambda \varphi(t_1) + (1-\lambda) \varphi(t_2) = 0$. For that we may choose,
    $\lambda = \frac{t+1}{2}$, $t_1 = 1$, $t_2 = -1$. We there conclude that $C\varphi(t) = \Phi((|t|-1)_{+})$.  
\end{proof}

\begin{remark}
Applying Theorem \ref{thm:LocalGammaConvergence} and Theorem \ref{thm:BoundsOnPeridynamicEnergy} to the function $\bv(\bx) = \bbF \bx$, 
we have for every matrix $\bbF \in \bbR^{d \times d}$
\begin{equation}\label{eq:BoundsOnIntegrand}
	\fint_{\bbS^{d-1}} \Phi((|\bbF \bsomega|-1)_+) \, \rmd \sigma(\bsomega) \leq f_\infty(\bbF) \leq Q\wt{\Phi}(\bbF)\,,
\end{equation}
where $\wt{\Phi}$ is given by \eqref{eq:BoundsOnIntegrandDefns2}. 

As a consequence, we see that, for matrices $\bbF\in \bbR^{d}$ such that $\bbF^T \bbF > \bbI$, we have $f_\infty(\bbF) = \wt{\Phi}(\bbF)$. 
Indeed, for such a matrix $\bbF$,  $|\bbF \bsomega| >1$, and therefore, $\Phi((|\bbF\bsomega| - 1)_+) = \Phi(|\bbF\bsomega| - 1)$ for all $\bsomega\in \bbS^{d-1}$.  It then follows from the inequality \eqref{eq:BoundsOnIntegrand} that 
\[
\begin{split}
    \wt{\Phi}(\bbF)&= \fint_{\bbS^{d-1}} \Phi(|\bbF \bsomega|-1) \, \rmd \sigma(\bsomega)
    =\fint_{\bbS^{d-1}} \Phi((|\bbF \bsomega|-1)_+) \, \rmd \sigma(\bsomega)\leq f_\infty(\bbF) \leq Q\wt{\Phi}(\bbF) \leq \wt{\Phi}(\bbF)\,
\end{split}
\]
where we used the fact that $Q \wt{\Phi} \leq \wt{\Phi}$. 

In the case of general nonlinear strain $s_m[\bv](\by,\bx)$, one can obtain an analogue of \eqref{eq:BoundsOnIntegrand} for $f_\infty^m$ in the same way:
    \begin{equation}\label{eq:BoundsOnGammaLimit:m}
	\fint_{\bbS^{d-1}} \Phi(m^{-1}(|\bbF \bsomega|^m-1)_+) \, \rmd \sigma(\bsomega) \leq f_\infty^m(\bbF) \leq Q \left[ \fint_{\bbS^{d-1}}  \Phi( m^{-1}(|\bbF \bsomega|^m - 1))  \, \rmd \sigma(\bsomega)  \right] \,.
\end{equation}
\end{remark}

\subsection{Invariances of the limiting density}

\begin{theorem}\label{thm:InvarianceOfIntegrand}
	The integrand $f_\infty$ obtained in Theorem \ref{thm:LocalGammaConvergence} is frame-indifferent, i.e.
	\begin{equation*}
		f_\infty(\bbU \bbF) = f_\infty(\bbF)
	\end{equation*}
	for any $\bbF \in \bbR^{d \times d}$ and any $\bbU \in \cO(d)$.
\end{theorem}

\begin{proof}
	Define $\bv(\bx) = \bbF \bx$ and $\bw(\bx) = \bbU \bbF \bx$. It suffices to show that
	\begin{equation}\label{eq:FrameIndiff:Pf1}
	\overline{F}_{\infty}(\bv) = \overline{F}_{\infty}(\bw)
	\end{equation}
	and the result will follow since $\grad \bv$ and $\grad \bw$ are both constant.
	To prove \eqref{eq:FrameIndiff:Pf1} note that by the definition of $F_n$
	\begin{equation*}
		F_n(\bu)=F_n(\bbU \bu)\,, \qquad \bu \in W^{1,p}(\Omega;\bbR^d)\,, \bbU \in \cO(d) \text{ a constant matrix.}
	\end{equation*}
	Let $\{ \bv_n \}_n$ converge strongly to $\bv$ in $L^p(\Omega;\bbR^d)$. Then $\bbU \bv_n$ converges to $\bw$ strongly in $L^p(\Omega;\bbR^d)$, and $I_n(\bv_n) = I_n(\bbU \bv_n)$. Thus 
	\begin{equation*}
		\liminf_{n \to\infty} I_n(\bv_n) = \liminf_{n \to\infty} I_n(\bbU \bv_n) \geq I_\infty(\bw)\,.
	\end{equation*}
	$I(\bw)$ is a lower bound on $\liminf_{n \to \infty} I_n(\bv_n)$ for any sequence $\bv_n$ converging to $\bv$ strongly in $L^p(\Omega;\bbR^d)$, so by definition of $\Gamma$-$\liminf$ we have $I_\infty(\bv) \geq I_\infty(\bw)$. Repeat the argument with a sequence $\bw_n$ converging to $\bw$ (then $\bbU^T \bw_n$ converges to $\bv$, etc.) to arrive at $I_\infty(\bv) \leq I_\infty(\bw)$.
\end{proof}

\begin{corollary}
	There exists a function $g_\infty : \mathrm{Sym}_d^+ \to \bbR$, where $\mathrm{Sym}_d^+$ is the space of $d \times d$ symmetric nonnegative definite matrices, such that for any $\bbF\in \mathbb{R}^{d\times d}$ 
	\begin{equation*}
		f_\infty(\bbF) = g_\infty(\bbF^T \bbF)\,.
	\end{equation*}
\end{corollary}

\subsection{The zero set of the limiting density}
    
We now use \eqref{eq:BoundsOnIntegrand} to characterize the zero set of the integrand $f_\infty.$ 

\begin{theorem}\label{thm:ZeroSetOfFxnal}
	$f_\infty(\bbF) = 0$ if and only if $\bbF^T \bbF \leq \bbI$ in the sense of quadratic forms.
\end{theorem}
\begin{proof}
	Suppose $f_\infty(\bbF) = 0$. Then the lower bound $\fint_{\bbS^{d-1}} \Phi((|\bbF \bsomega|-1)_+) \, \rmd \sigma(\bsomega)$ in \eqref{eq:BoundsOnIntegrand} 
 must also be zero. However, the lower bound is zero so long as the integrand 
 $\Phi \big( ( |\bbF\bsomega| - 1)_+ \big) =0$ if and only if $|\bbF\bsomega| \leq 1$ for all $\bsomega\in \bbS^{d-1}$. It then follows that $|\bbF\bsomega| \leq 1$, so $|\bbF \bsomega|^2 \leq 1$, so $\Vint{(\bbI - \bbF^T \bbF) \bsomega, \bsomega } \geq 0$ for every $ \bsomega \in \bbS^{d-1}$. Thus $ \bbF^T \bbF \leq \bbI$. 
 
 Conversely,  suppose $\bbF\in \mathbb{R}^d$ such that $\bbF^T \bbF \leq \bbI$. We claim that the upper bound $Q\wt{\Phi}$ in \eqref{eq:BoundsOnIntegrand} is zero for such an $\bbF$ and therefore $f_\infty(\bbF)=0$.
 To show this, we use a convex integration scheme from \cite{ledret1995nonlinear}.
	It is clear from the definition and from the fact that $\Phi(0) = 0$ that
	\begin{equation*}
			\fint_{\bbS^{d-1}} \Phi( | |\bbF \bsomega|-1 |) \, \rmd \sigma(\bsomega) = 0 \text{ for } \bbF^T \bbF = \bbI\,.
	\end{equation*}
	Let  $0 \leq \lambda_1 \leq \ldots \leq \lambda_d \leq 1$ be the singular values of $\bbF$. Then there exist $\bbU'$,$\bbU'' \in \cO(d)$ such that
	\begin{equation*}
		\bbF = \bbU'' \sqrt{\bbF^T \bbF} = \bbU'' (\bbU')^T \bbD (\bbU')\,,
	\end{equation*}
 where $\bbD  =\diag(\lambda_1, \ldots, \lambda_d).$ 
It then follows that 	
 \begin{equation*}
	\wt{\Phi}(\bbF)= \fint_{\bbS^{d-1}} \Phi(  | |\bbF \bsomega|-1 |) \, \rmd \sigma(\bsomega)  =  \fint_{\bbS^{d-1}} \Phi ( | |\diag(\lambda_1, \ldots, \lambda_d) \bsomega|-1 |) \, \rmd \sigma(\bsomega)\, = \wt{\Phi}(\bbD),
	\end{equation*}
and as a consequence, by the definition of quasiconvexity, that 
 \[
 Q\wt{\Phi}(\bbF) = Q\wt{\Phi}(\bbD) \quad\quad\text{for any $\bbF\in \bbR^{d\times d}$}. 
 \]
	Thus it suffices to show $Q\wt{\Phi}(\bbF) = 0$ 
 for matrices $\bbF = \diag(\lambda_1, \ldots, \lambda_d)$ with $0 \leq \lambda_1 \leq \ldots \leq \lambda_d \leq 1$. To that end, let $B \subset \bbR^d$ be the unit ball. For $i \in \{1, \ldots, d\}$ define the continuous periodic function $\gamma_i : \bbR \to \bbR$ by
	\begin{equation*}
		\gamma_i(t) := 
		\begin{cases}
			(1-\lambda_i) t\,, &\quad \text{ if } 0 \leq t - \lfloor t \rfloor \leq \frac{1+\lambda_i}{2}\,; \\
			(-1-\lambda_i) (t-1)\,, &\quad \text{ if }  \frac{1+\lambda_i}{2} \leq t - \lfloor t \rfloor \leq 1\,.
		\end{cases}
	\end{equation*}
	Set $\bv(\bx) = \bbF\bx = (\lambda_i x_i)_i$, so that $\grad \bv(\bx) = \bbF$, and set
	\begin{equation*}
		\bv_k(\bx) := \bv(\bx) + \frac{1}{k} (\gamma_i(kx_i))_i = 
			\begin{bmatrix}
				\lambda_1 x_1 + \frac{1}{k} \gamma_1(kx_1) \\
				\lambda_2 x_2 + \frac{1}{k} \gamma_2(kx_2) \\
				\cdots\\
				\lambda_d x_d + \frac{1}{k} \gamma_d(kx_d)
			\end{bmatrix}\,.
	\end{equation*}
	Then $\bv_k \xrightharpoonup{*} \bv$ weak-$*$ in $W^{1,\infty}(B;\bbR^d)$ as $k \to \infty$. On the other hand,
	\begin{equation*}
		\grad \bv_k(\bx) = \diag(\veps_{k,1}(x_1), \ldots, \veps_{k,d}(x_d))\,,
	\end{equation*}
	where for $i \in \{1, \ldots, d\}$
	\begin{equation*}
		\veps_{k,i}(x_i) := 
		\begin{cases}
			1\,, &\quad \text{ if } 0 \leq kx_i - \lfloor k x_i \rfloor \leq \frac{1+\lambda_i}{2}\,; \\
			-1\,, &\quad \text{ if }  \frac{1+\lambda_i}{2} \leq k x_i - \lfloor k x_i \rfloor \leq 1\,.
		\end{cases}
	\end{equation*}
	Thus 
	\begin{equation*}
		(\grad \bv_k)^T (\grad \bv_k) = \bbI \text{ almost everywhere in } B\,,
	\end{equation*}
	and so
	\begin{equation*}
		\wt{\Phi}(\nabla \bv_k (\bx)) = \fint_{\bbS^{d-1}} \Phi(| |\grad \bv_k(\bx) \bsomega|-1 |) \, \rmd \sigma(\bsomega) = 0 \text{ for all } k \text{ and for almost every } \bx \in B\,.
	\end{equation*}
	Therefore, by the $W^{1,\infty}$ weak-$*$ lower semicontinuity of quasiconvex functions, using the fact that $Q\wt{\Phi}(\bbF)\leq \wt{\Phi}(\bbF)$, 
	\begin{equation*}
		\begin{split}
			Q\wt{\Phi}(\bbF)
			= \frac{1}{|B|} \int_{B} Q\wt{\Phi}(\nabla \bv(\bx))\, \rmd \bx&\leq \liminf_{n\to \infty}\frac{1}{|B|} \int_{B} Q\wt{\Phi}(\nabla \bv_k(\bx))\, \rmd \bx\\
   &\leq \liminf_{n\to \infty}\frac{1}{|B|} \int_{B} \wt{\Phi}(\nabla \bv_k(\bx))\, \rmd \bx = 0\,.
		\end{split}
	\end{equation*}
 That concludes the proof. 
\end{proof}

\subsection{Coercivity of the limiting density}

\begin{theorem}
    Suppose that $\Phi$ satisfies \eqref{eq:PDEnergy:ConditionsOnG}. Then there exists a constant $C(\Phi,p,d) > 0$ such that
    \begin{equation}\label{eq:LowerBoundOnGammaLowerBound}
        \fint_{\bbS^{d-1}} \Phi((|\bbF \bsomega|-1)_+) \, \rmd \sigma(\bsomega) \geq C (|\bbF|^p-1)
    \end{equation}
    for all $\bbF \in \bbR^{d \times d}$.
\end{theorem}

\begin{proof}
    Then by an argument similar to that in the proof of Theorem \ref{thm:ZeroSetOfFxnal} it suffices to show
    \eqref{eq:LowerBoundOnGammaLowerBound} with $\bbF$ of the form $\bbD := \diag(\lambda_1,\lambda_2,\ldots,\lambda_d)$, where $0 \leq \lambda_1 \leq \ldots \leq \lambda_d$.
    Using the equivalence of norms $|\bbF| \leq C(d) |\lambda_d|$ as well as the lower bound on $\Phi$, \eqref{eq:LowerBoundOnGammaLowerBound} will be established if we show that
    \begin{equation*}
        \fint_{\bbS^{d-1}} (|\bbD \bsomega|-1)_+^p \, \rmd \sigma(\bsomega) \geq C (|\lambda_d|^p-1)\,.
    \end{equation*}
    Clearly, we may assume that $\lambda_d \geq 2$.
    Let $J := \{ \bsomega \in \bbS^{d-1} \, : \, \omega_d > 1/2 \}$. Then $|\bbD \bsomega| \geq \lambda_d |\omega_d| \geq 1$ on $J$, and since $t \mapsto (|t| - 1)_+^p$ is monotone 
    \begin{equation*}
    \begin{split}
        \fint_{\bbS^{d-1}} (|\bbD \bsomega|-1)_+^p \, \rmd \sigma(\bsomega) \geq \fint_{\bbS^{d-1}} (|\lambda_d \omega_d|-1)_+^p \, \rmd \sigma(\bsomega) 
        &\geq \frac{1}{|\bbS^{d-1}|} \int_{J} (|\lambda_d \omega_d|-1)^p \, \rmd \sigma(\bsomega) \\
        &\geq C(d,p) \left( \int_{J} (|\lambda_d \omega_d|^p \, \rmd \sigma(\bsomega) - 1 \right) \\
        &= C(d,p) (|\lambda_d|^p -1 )\,.
    \end{split}
    \end{equation*}
\end{proof}

\begin{corollary}\label{cor:CoercivityOfIntegrand}
    Let $f_{\infty} :\bbR^{d \times d} \to \bbR$ be as in Theorem \ref{thm:LocalGammaConvergence}.
    Then there exists $C = C(\Phi,d,p) >0 $ such that
    \begin{equation*}
        C (|\bbF|^{p} - 1) \leq f_{\infty}(\bbF) \qquad \forall \bbF \in \bbR^{d \times d}\,. 
    \end{equation*}
\end{corollary}

\subsection{Summary of results for generalized strains}
Analogues of the results from this section hold for the quasiconvex functional $\overline{F}_\infty^m$ given via integral representation in Remark \ref{gamma-conv-remark}. The following theorem can be obtained using similar techniques:

\begin{theorem}
	For $m \geq 1$ and $A \in \cA_0(\Omega)$, define the functionals $\overline{F}_n^m(\cdot,A) : L^{mp}(\Omega;\bbR^d) \to \overline{\bbR}$ to be extensions of $F_n^m(\cdot,A)$, defined as
    \begin{equation*}
        F_{n}^{m}(\bv,A):=\int_A \int_{A} \rho_n(\bx-\by) \Phi(|s_m[\bv](\bx,\by)|) \, \rmd \by \, \rmd \bx\,,
    \end{equation*}
    to $L^{mp}(\Omega;\bbR^d)$; that is,
	\begin{equation*}
		\overline{F}_n^m(\bv,A) := \begin{cases}
			F_n^m(\bv,A)\,, & \bv \in \frak{W}^{\rho_n,mp}(\Omega;\bbR^d) \\
			+\infty\,, & \text{ otherwise. }
		\end{cases}
	\end{equation*}
	Then there exists a quasiconvex Carath\'eodory function $f_\infty^m : \bbR^{d \times d} \to \bbR$ satisfying the growth condition
	\begin{equation*}
		C(|\bbF|^{mp} - 1) \leq f_\infty(\bbF) \leq C (|\bbF|^{mp} + 1)
	\end{equation*}
	such that $\overline{F}_n^m(\cdot,A)$ $\Gamma$-converges in the strong topology on $L^{mp}(\Omega;\bbR^d)$ to $\overline{F}_\infty^m(\cdot,A)$, where the limiting functional $\overline{F}_{\infty}^m(\cdot,A) : L^{mp}(\Omega;\bbR^d) \to \overline{\bbR}$ is defined as
	\begin{equation*}
		\overline{F}_{\infty}^m(\bv,A) = \begin{cases}
			\intdm{A}{f_\infty^m(\grad \bv(\bx))}{\bx}\,, & \bv \in W^{1,mp}(\Omega;\bbR^d)\\
			+\infty\,, & \text{ otherwise. }
		\end{cases}
	\end{equation*}
	Moreover, the effective strain energy density $f_\infty^m$ is frame-indifferent, and $f_\infty^m(\bbF) = 0$ if and only if $\bbF^T \bbF \leq \bbI$ in the sense of quadratic forms.
\end{theorem}

\begin{remark}
    In general, $f_\infty^m$ cannot be obtained in a straightforward way from the upper and lower bounds in \eqref{eq:BoundsOnGammaLimit:m}, as the two do not coincide. Take $d = 2$, $\Phi(t) = t^2$, $m = 2$, and let $\bbF$ be any matrix in $\bbR^{2 \times 2}$.
    Then 
    $$
    \wt{\Phi}(\bbF) = \frac{1}{4} \fint_{\bbS^{1}} (|\bbF \bsomega|^2 - 1)^2 \, \rmd \sigma(\bsomega) = \frac{1}{32} \left( |\bbF^T \bbF - \bbI|^2 + \frac{1}{2} (|\bbF|^2 - 2)^2 \right)\,.
    $$
    The quasiconvexification of this function is computed explicitly in \cite[Theorem 6.29]{Dacorogna}. From that formula it is clear that $Q \wt{\Phi}(\bbF) = \wt{\Phi}(\bbF)$ for all $\bbF$ with norm large enough, i.e. $|\bbF| > 8$. However, $\frac{1}{4} \fint_{\bbS^{1}} ( |\bbF \bsomega|^2 - 1)_+^2 \, \rmd \sigma(\bsomega) \neq \wt{\Phi}(\bbF)$ for any $\bbF$ of the form $\diag(\lambda,1/\lambda)$ for any $\lambda$ large.

    We note that in the special case $d = 1$, it is straightforward to check that the upper and lower bounds \eqref{eq:BoundsOnGammaLimit:m} on $f_\infty$ coincide, which gives the precise $\Gamma$-limit \begin{equation*}
		\overline{F}_\infty^m(v,A) = \begin{cases}
			\intdm{A}{ \Phi(m^{-1} (|v'(x)|^m - 1)_+)}{x}\,, &v \in W^{1,mp}(\Omega) \\
			+ \infty\,, &\text{ otherwise. }
		\end{cases}
    \end{equation*}
\end{remark}

\appendix
\section{Characterizations of distance-preserving maps}
\begin{theorem}\label{lma:MeasFxnsAreRigid}
	Let $\Omega \subset \bbR^d$ be a bounded domain. Suppose $
	\bv : \Omega \to \bbR^d$ is measurable, and suppose $\bv$ satisfies
	\begin{equation}\label{eq:IsometryCondition}
		|\bv(\bx)-\bv(\by)| = |\bx-\by|
	\end{equation}
	for $\cL^{2d}$-almost every $(\bx,\by) \in \Omega \times \Omega$, where in general $\cL^N$ denotes $N$-dimensional Lebesgue measure. 
	Then there exists a constant matrix $\bbF \in \cO(d)$ and $\bfb \in \bbR^d$ such that $\bv(\bx) = \bbF \bx+ \bfb$ for almost every $\bx \in \Omega$.
\end{theorem}

\begin{proof}
We prove the result first under the additional assumption that $\bv : \Omega \to \bbR^d$ is continuous.
	Then the functions $f_1(\bx,\by) := |\bv(\bx)-\bv(\by)|$ and $f_2(\bx,\by):=|\bx-\by|$ are continuous on $\Omega \times \Omega$. Let $X \subset \Omega \times \Omega$ be the set where \eqref{eq:IsometryCondition} holds. Then $(\Omega \times \Omega) \backslash X$ is dense in $\Omega\times \Omega$. So by density and continuity we have 
	\[
		|\bv(\bx)-\bv(\by)| = |\bx-\by| \quad \forall\, \bx\,, \by \in \Omega \times \Omega\,.
	\]
	Since the relation \eqref{eq:IsometryCondition} is translation- and shift-invariant, we can assume without loss of generality that $B({\bf 0},R) \subset \Omega$ for some $R>0$ and $\bv({\bf 0}) = {\bf 0}$. We will show that $\bv(\bx) = \bbF \bx$ for some constant matrix $\bbF \in \cO(d)$, obtaining the result for continuous functions.
	
	First, by \eqref{eq:IsometryCondition} $|\bv(\bx)|^2 = |\bv(\bx)-\bv({\bf 0})|^2 = |\bx-{\bf 0}|^2 = |\bx|^2$.
	The identity
	\begin{equation}\label{eq:IsometryCondition:InnerProd}
		\Vint{\bv(\bx),\bv(\by)} = \Vint{\bx,\by}
	\end{equation}
	for every $(\bx,\by) \in \Omega \times \Omega$ follows, since
	\begin{equation*}
		\begin{split}
			|\bx|^2 - 2 \Vint{\bx,\by} + |\by|^2 =|\bx-\by|^2 
			= |\bv(\bx)-\bv(\by)|^2
			&= |\bv(\bx)|^2 - 2 \Vint{\bv(\bx),\bv(\by)} + |\bv(\by)|^2 \\
			&= |\bx|^2 - 2\Vint{\bv(\bx),\bv(\by)} + |\by|^2\,.
		\end{split}
	\end{equation*}
	Define the $d \times d$ matrix
	\begin{equation*}
		\bbF = (a^{jk}) = \frac{1}{R} (u_j( R \be_k))\,,
	\end{equation*}
	where $\be_k$ is the vector in $\bbR^d$ with $k$th coordinate $1$ and all other coordinates $0$.
	Using this definition and using \eqref{eq:IsometryCondition:InnerProd} with $R \be_j$ in place of $\by$,
	\begin{equation*}
		[\bbF^T \bv(\bx)]_j = \sum_{k=1}^{d} a^{kj} u_k(\bx) = \sum_{k=1}^{d}  \frac{1}{R} u_k(R \be_j) u_k(\bx) = \frac{1}{R} \Vint{\bv(\bx),\bv( R \be_j)} = \frac{1}{R}  \Vint{\bx,R \be_j} = x_j\,.
	\end{equation*}
	This is true for any $j \in \{ 1, \ldots, d\}$, and so
	\begin{equation*}
		\bbF^T \bv(\bx) = \bx\,.
	\end{equation*}
	Again using the definition of $\bbF$ and using \eqref{eq:IsometryCondition:InnerProd} with $(R \be_j, R \be_k)$ in place of $(\bx,\by)$,
	\begin{equation*}
		[\bbF^T \bbF]_{jk} = \sum_{\ell = 1}^n a^{\ell j} a_{\ell k} = \frac{1}{R^2} \sum_{\ell = 1}^n u_{\ell}(R \be_j) u_{\ell} (R \be_k) = \frac{1}{R^2}\Vint{\bv(R \be_j),\bv(R \be_k)} = \frac{1}{R^2}\Vint{R \be_j, R\be_k} = \delta_{jk}\,.
	\end{equation*}
	Therefore $\bbF^T \bbF = \bbI$, and we conclude that $\bv(\bx) = \bbF \bx$ with $\bbF \in \cO(d)$.
	
	Now suppose $\bv$ is measurable. By Lusin's theorem for every $n \in \bbN$ there exists a closed set $K_n \subset \Omega$ with $\cL^d(\Omega \setminus K_n) < \frac{1}{n}$ such that $\bv$ is continuous on $K_n$. By \eqref{eq:IsometryCondition} $\bv$ is also Lipschitz on $K_n$, with Lipschitz constant $1$. By Kirszbraun's theorem there exists a function $\bv_n : \bbR^d \to \bbR^d$ that is $1$-Lipschitz and coincides with $\bv$ on $K_n$. Therefore by Rellich's theorem there exists a subsequence (not relabeled) $\{\bv_n\}$ that converges uniformly on $\overline{\Omega}$ to a continuous and $1$-Lipschitz function $\wt{\bv}$. By definition of the $\bv_n$ it follows that $\bv = \wt{\bv}$ almost everywhere on $\Omega$. Therefore $\bv$ has a Lipschitz (hence continuous) representative, and the first part of the proof applies.
\end{proof}

This rigidity result can be strengthened in the spirit of \cite{reshetnyak1967liouville}, as we demonstrate in the next theorem.

\begin{theorem}\label{thm:ReshetnyakVersion1}
    Let $m \geq 1$, let $\rho \in L^1(\bbR^d)$ be a nonnegative radial kernel satisfying $B(\boldsymbol{0},r) \subset \supp \rho \subset B(\boldsymbol{0},R)$ for given $0 < r < R$, and let $\Phi$ be a nondecreasing convex function satisfying \eqref{eq:PDEnergy:ConditionsOnG}. Suppose a sequence of vector fields $\{ \bv_n \}_n \subset \frak{W}^{\rho,p}(\Omega;\bbR^d)$ satisfies
	\begin{equation*}
		\lim\limits_{n \to \infty} \iintdm{\Omega}{\Omega}{\rho(\bx-\by) \Phi \left( |s_m[\bv_n](\by,\bx)| \right) }{\by}{\bx} = 0\,.
	\end{equation*}
	Suppose additionally that there exists a function $\bv \in L^{1}(\Omega;\bbR^d)$ such that $\bv_n \to \bv$ in $L^{1}(\Omega;\bbR^d)$. Then there exists a constant matrix $\bbF \in \cO(d)$ and a vector $\bfb \in \bbR^d$ such that $\bv(\bx) = \bbF \bx + \bfb$ for almost every $\bx \in \Omega$.
\end{theorem}

\begin{proof}
	Since $\bv_n \to \bv$ in $L^{1}(\Omega;\bbR^d)$ there exists a subsequence (not relabeled) $\{\bv_n\}_n$ that converges to $\bv$ $\cL^d$-almost everywhere in $\Omega$. Since $\Phi$ is continuous,
	\begin{equation*}
		\Phi \left( |s_m[\bv_n](\by,\bx)| \right)  \to \Phi \left( |s_m[\bv](\by,\bx)| \right)  \; \cL^{2d}\text{-a.e. in } \Omega \times \Omega\,,
	\end{equation*}
	and so by Fatou's lemma
	\begin{equation*}
		\iintdm{\Omega}{\Omega}{\rho(\bx-\by) \Phi \left( |s_m[\bv](\by,\bx)| \right) }{\by}{\bx} 
		\leq \liminf_{n \to \infty} \iintdm{\Omega}{\Omega}{\rho(\bx-\by) \Phi \left( |s_m[\bv_n](\by,\bx)| \right) }{\by}{\bx} = 0\,.
	\end{equation*}
	Therefore it must be that
        $
		\rho(\by-\bx)\Phi \left( |s_m[\bv](\by,\bx)| \right)  = 0 \text{ for a.e. } \by \in \supp \rho + \bx\,, \bx \in \Omega\,.
        $
	For any $\bx \in \Omega$, define $r_{\bx} = \min\{r,\dist(\bx,\p \Omega)\}$. Then 
	by assumption on $\rho$, 
	\begin{equation*}
		\Phi \left( |s_m[\bv](\by,\bx)| \right) = 0 \qquad \text{ for a.e. } \by \in B(\bx,r_{\bx})\,, \bx \in \Omega\,.
	\end{equation*}
	Now fix $\bx_0 \in \Omega$ and fix $r_0 = r_{\bx_0}$. Then by definition of $\Phi$ and $s_m$
	\begin{equation*}
		|\bv(\by)-\bv(\bx)|=|\by-\bx| \qquad \text{ for a.e. } \by \in B(\bx_0,r_0/2)\,, \bx \in B(\bx_0,r_0/2)\,.
	\end{equation*}
	The relation then holds for $\cL^{2d}$-a.e. $(\bx,\by) \in B(\bx_0,r_0/2) \times B(\bx_0,r_0/2)$, and so Lemma \ref{lma:MeasFxnsAreRigid} applies on $B(\bx_0,r_0/2)$.
    
    By covering $\Omega$ with sets of the form $B(\bx_0,r_0/2)$, we see that $\bv$ is a possibly piecewise affine map on $\Omega$.
	To conclude, note that since $\Omega$ is a domain there exists a finite chain of sets of the form $B(\bx_0,r_0/2)$ between any $\bx_1$ and $\bx_2$ in $\Omega$, and thus $\bv$ must be the same affine map at both points.
\end{proof}

\begin{theorem}\label{thm:LocalRigiditySmoothFxns}
    Suppose that $\bv \in C^2(\overline{\Omega};\bbR^d)$ satisfies \eqref{eq:Intro:LocalLimit}. Then there exists a constant matrix $\bbF \in \cO(d)$ and a vector $\bfb \in \bbR^d$ such that $\bv(\bx) = \bbF \bx + \bfb$ for every $\bx \in \Omega$.
\end{theorem}

\begin{proof}
    First, since $\bv \in C^2\overline{\Omega};\bbR^d)$, $\det \grad \bv \in C^1(\overline{\Omega};\bbR^d)$, with
    \begin{equation*}
        \p_\ell [\det \grad \bv(\bx)] = \det \grad \bv(\bx) \sum_{j,k =1}^d \p_j u_k \cdot \p_\ell [\p_j u_k ] = \frac{\det \grad \bv(\bx) }{2} \p_\ell [ |\grad \bv(\bx)|^2 ]\,.
    \end{equation*}
    But $|\grad \bv(\bx)|^2 = \mathrm{tr}( \grad \bv^T \grad \bv) = d$, so therefore $\det \grad \bv$ is constant in all of $\overline{\Omega}$. Thus, either $\grad \bv = \mathrm{cof} \grad \bv$ for all $\bx \in \Omega$ or $\grad \bv = -\mathrm{cof} \grad \bv$ for all $\bx \in \Omega$. In both cases it follows from the Piola identity $\div \mathrm{cof} \grad \bv = {\bf 0}$ that $\bv$ is a harmonic function on $\Omega$. Thus $\bv \in C^{\infty}(\Omega)$, and so we can compute
    \begin{equation*}
        0 = \frac{1}{2} \Delta [ |\grad \bv|^2 - d] = \grad \bv : \Delta[ \grad \bv ] + |\grad^2 \bv|^2 = |\grad^2 \bv|^2\,.
    \end{equation*}
    Thus $\grad \bv(\bx)$ is constant in $\Omega$, and necessarily belongs to $\cO(d)$.
\end{proof}

\section*{Acknowledgements}

This manuscript has benefited from discussions with Qiang Du, Marta Lewicka, Armin Schikorra, and Xiaochuan Tian. The authors thank them for their valuable insight and input.


\begin{thebibliography}{999}

\bibitem{aguiar}
 Aguiar, A.R., Royer--Carfagni, G. F.,   Seitenfuss, A. B.: 
\newblock Wiggly strain localizations in peridynamic bars with non-convex potential.
\newblock {International Journal of Solids and Structures 138 1--12}(2018). 
\bibitem{alali2015peridynamics}
Alali, B., Gunzburger, M.:
\newblock Peridynamics and material interfaces.
\newblock Journal of Elasticity, 120(2):225--248 (2015).

\bibitem{alicandro2020variational}
Alicandro, R., Ansini, N.,  Braides, B., Piatnitski, A.,
  Tribuzio, A.:
\newblock A variational theory of convolution-type functionals, Springer Singapore (2020). https://doi.org/10.1007/978-981-99-0685-7 

\bibitem{alicandro2004general}
Alicandro, R., Cicalese, M.:
\newblock A general integral representation result for continuum limits of
  discrete energies with superlinear growth.
 SIAM journal on mathematical analysis, 36(1):1--37, (2004).
\bibitem{Bellido2015}
Bellido, J. C., Mora-Corral, C., Pedregal, P.:
\newblock{
Hyperelasticity as a $\Gamma$-limit of peridynamics when
the horizon goes to zero}
\newblock{Calculus Variations PDEs}
\newblock{54 1643--70, (2015).}
\bibitem{bellido2020bond}
Bellido,J.~C.,  Cueto, J.,  Mora-Corral, C.: 
\newblock Bond-based peridynamics does not converge to hyperelasticity as the
  horizon goes to zero.
\newblock { J. Elasticity}, 141(2):273--289, (2020).

\bibitem{BBM}
Bourgain, J.,  Brezis, H., Mironescu, P.:
\newblock Another look at {S}obolev spaces.
\newblock In Jos{\'e}~Luis Menaldi, Edmundo Rofman, and Agnes Sulem, editors,
  { Optimal Control and Partial Differential Equations: In Honour of
  Professor Alain Bensoussan's 60th Birthday}, pages 439--455. IOS Press, (2001).

\bibitem{braides1998homogenization}
Braides, A., Defranceschi, A.:
\newblock { Homogenization of multiple integrals}. \newblock Number~12. Oxford University Press, (1998).
\bibitem{Braides-DalMaso} 
Braides, A., Maso, G. D.: Compactness for a class of integral functionals with interacting local and non-local terms. Calc. Var. 62, 148 (2023). https://doi.org/10.1007/s00526-023-02491-w


\bibitem{Dacorogna} Dacorogna, B.: 
\newblock Introduction to the Calculus of Variations. Imperial College Press,(2014). 

\bibitem{DalMaso} {Dal Maso, G.:} 
\newblock {An introduction to $\Gamma$-convergence. Progress in Nonlinear Differential Equations and their Applications}
\newblock{Birkhauser Boston, Inc., Boston, MA, (1993).}
\bibitem{DalMaso-linearization}
\newblock Dal Maso, G., Negri, M., Percivale, D.: Linearized elasticity as  $\Gamma$-limit of finite elasticity. Set Valued Anal. 10, 165--183 (2002).
\bibitem{Kaushik}
 Dayal, D., Bhattacharya, K.: 
\newblock Kinetics of phase transformations in the peridynamic formulation of continuum mechanics.
\newblock { Journal of the Mechanics and Physics of Solids},
\newblock Volume 54, Issue 9, Pages 1811-1842, (2006). 
\bibitem{Du-Mengesha-Tian-compactness}
Du, Q., Mengesha, T., Tian, X.: 
\newblock{Nonlocal criteria for compactness in the space of Lp vector fields,}
\url{https://arxiv.org/abs/1801.08000}
\bibitem{Q-T-T}
Du, Q., Tao, Y., Tian, X.: 
\newblock A Peridynamic Model of Fracture Mechanics with Bond-Breaking. 
\newblock {\it J Elast} 132, 197--218 (2018). 
\newblock https://doi.org/10.1007/s10659-017-9661-2
\bibitem{Evans}
Evans, L. C.: Partial differential equations. 
American Mathematical Society, Providence, R.I., (2010). 

\bibitem{fonseca2007modern}
Fonseca, I.,  Leoni, G.:
\newblock {Modern Methods in the Calculus of Variations: {$L^p$} Spaces}.
\newblock Springer Science \& Business Media, (2007).


\bibitem{ledret1995nonlinear}
Le~Dret, H., Raoult, A.:
\newblock The nonlinear membrane model as variational limit of nonlinear
  three-dimensional elasticity.
\newblock { J. Math. Pures Appl. (9)}, 74(6):549--578, (1995).


\bibitem{Lipton-Brittle}Lipton, R.
\newblock Dynamic brittle fracture as a small horizon limit of peridynamics.
\newblock { J. Elast.} 117(1), 21--50 (2014).

\bibitem{Lipton-Cohesive}
Lipton R.: 
\newblock Cohesive Dynamics and Brittle Fracture 
\newblock { J. Elast.} 124, pages 143--191 (2016).
 

\bibitem{mengesha-du-2014-elasticity}
Mengesha, T., Du, Q.: The peridynamic system as a nonlocal boundary value problem. Journal of Elasticity, 116, 27--51, (2014).
\bibitem{mengesha2015VariationalLimit}
Mengesha, T.,  Du, Q.:
\newblock On the variational limit of a class of nonlocal functionals related
  to peridynamics.
\newblock { Nonlinearity}, 28(11):3999, (2015).


\bibitem{Ponce-poincare}
Ponce, A.C.: 
\newblock An estimate in the spirit of Poincare’s inequality.
\newblock{J. Eur. Math. Soc. , 6,1,1--15, (2004).}
\bibitem{ponce2004new}
Ponce, A.C.:
\newblock A new approach to sobolev spaces and connections to
  $\gamma$-convergence.
\newblock { Calc. Var. Partial Differential Equations}, 19(3):229--255,
  (2004).

\bibitem{reshetnyak1967liouville}
Reshetnyak, Y.~G.:
\newblock {L}iouville's conformal mapping theorem under minimal regularity
  hypotheses.
\newblock { Sibirsk. Mat. \v{Z}.}, 8:835--840, (1967).
  
\bibitem{SL2008}
Silling, S. A. and Lehoucq, R. B.:
\newblock Convergence of Peridynamics to Classical Elasticity Theory
\newblock {J. Elast.}
\newblock 93, 13--37, (2008). 
\bibitem{C-N}
Silling, S.A., Weckner, O., Askari, E. et al.: 
\newblock Crack nucleation in a peridynamic solid.
\newblock { Int J Fract} 162, 219--227 (2010). \newblock https://doi.org/10.1007/s10704-010-9447-z
\bibitem{Silling2001}Silling, S.A.:  Reformulation of elasticity theory for discontinuities and long-range forces.  {J. Mech. Phys. Solids}, 48, 175--209, (2000).

\bibitem{Silling2010} Silling, S.A.: Linearized theory of peridynamic states. { J. Elast.} 99, 85--111, (2010).
\bibitem{Silling2007} Silling, S.A., Epton, M.,   Weckner, O.,   Xu, J., Askari, E.: Peridynamic states and constitutive modeling.
J. Elast. 88, 151--184, (2007).

\end{thebibliography}
\end{document}